\documentclass[11pt, a4paper]{amsart}
\usepackage{amsmath}
\usepackage{amssymb}
\usepackage{color}

\newtheorem{prop}{Proposition}[section]
\newtheorem{teo}{Theorem}[section]
\newtheorem{lema}{Lemma}[section]
\newtheorem{coro}{Corollary}[section]

\theoremstyle{definition}

\def\ep{\varepsilon}

\def\a{\mathfrak q}
\def\R{\mathbb R}

\def\X{{\mathcal X}}
\def\H{{\mathcal H}}
\def\l{\lambda}
\def\ul{u^\l}

\begin{document}
\title[One-dimensional nonlocal diffusion
in exterior domains]{Asymptotic behavior for a one-dimensional nonlocal diffusion equation in exterior domains}

\author[Cort\'{a}zar, Elgueta, Quir\'{o}s \and Wolanski]{C. Cort\'{a}zar, M. Elgueta, F. Quir\'{o}s \and N. Wolanski}

\address{Carmen Cort\'{a}zar\hfill\break\indent
Departamento  de Matem\'{a}tica, Pontificia Universidad Cat\'{o}lica
de Chile \hfill\break\indent Santiago, Chile.} \email{{\tt
ccortaza@mat.puc.cl} }

\address{Manuel Elgueta\hfill\break\indent
Departamento  de Matem\'{a}tica, Pontificia Universidad Cat\'{o}lica
de Chile \hfill\break\indent Santiago, Chile.} \email{{\tt
melgueta@mat.puc.cl} }

\address{Fernando Quir\'{o}s\hfill\break\indent
Departamento  de Matem\'{a}ticas, Universidad Aut\'{o}noma de Madrid
\hfill\break\indent 28049-Madrid, Spain.} \email{{\tt
fernando.quiros@uam.es} }

\address{Noemi Wolanski \hfill\break\indent
Departamento  de Matem\'{a}tica, FCEyN,  UBA,
\hfill\break \indent and
IMAS, CONICET, \hfill\break\indent Ciudad Universitaria, Pab. I,\hfill\break\indent
(1428) Buenos Aires, Argentina.} \email{{\tt wolanski@dm.uba.ar} }

\thanks{All authors supported by  FONDECYT grants 7090027 and 1110074. The third author supported by
the Spanish Project MTM2011-24696. The fourth author supported by
 CONICET PIP625,
Res. 960/12, ANPCyT PICT-2012-0153 and UBACYT X117.}



\date{}

\begin{abstract}
We study the long time behavior of solutions to the nonlocal diffusion equation $\partial_t u=J*u-u$  in an exterior  one-dimensional domain, with zero Dirichlet data on the complement.  In the far field scale, $\xi_1\le|x|t^{-1/2}\le\xi_2$, $\xi_1,\xi_2>0$, this  behavior is given by a multiple of the dipole solution for the local heat equation with a diffusivity determined by $J$. However, the proportionality constant is not the same on $\R_+$ and $\R_-$: it is given by the asymptotic first momentum of the solution on the corresponding half line, which can be computed in terms of the initial data. In the near field scale, $|x|\le t^{1/2}h(t)$,  $\lim_{t\to\infty}h(t)=0$, the solution scaled by a factor  $t^{3/2}/(|x|+1)$  converges  to a  stationary solution of the problem that behaves as $b^\pm{x}$ as $x\to\pm\infty$. The constants $b^\pm$ are obtained through a matching procedure with the far field limit.  In the very far field, $|x|{\ge}t^{1/2} g(t)$,  $g(t)\to\infty$, the solution has order $o(t^{-1})$.
\end{abstract}

\maketitle

\date{}

\section{Introduction}
\label{Intro} \setcounter{equation}{0}

Let  $\mathcal{H}\subset \mathbb{R}$ be a non-empty bounded open set, which may be assumed without loss of generality to satisfy
\begin{equation}
\label{hypotheses.H}
\tag{$H_\H$}
(-a_0,a_0)\subset \H\subset(-a,a),\qquad 0<a_0<a<\infty.
\end{equation}
We do not assume $\H$ to be connected, so it may represent one or several holes in an otherwise homogeneous medium. Our goal is to study the large-time behavior of the solution
to a certain nonlocal heat equation  in the exterior domain $\R\setminus\H$ with zero
data on the boundary, namely,
\begin{equation}\label{problem}
\begin{cases}
\displaystyle\partial_t u(x,t)=Lu(x,t)\ &\mbox{in } (\R\setminus\H)\times\R_+,\\
u(x,t)=0&\mbox{in }\mathcal{H}\times\R_+,\\
u(x,0)=u_0(x)&\mbox{in }\mathbb{R}.
\end{cases}
\end{equation}
The nonlocal operator $L$ is defined as $Lg:=J*g-g$, with a convolution kernel $J$ that  satisfies
\begin{equation}
\label{hypotheses.J}
\tag{$H_J$}
\left\{
\begin{array}{l}
J\in C^2_{\textrm{c}}(\mathbb{R}),\quad J\ge0, \quad \mathop{\rm supp} J=(-d,d),  \quad \int_\R J=1,\\[6pt]
J(x)=J(-x)\ \text{ for }x\in\R, \quad J(x_1)\ge J(x_2) \ \text{ if } 0\le x_1\le x_2.
\end{array}
\right.
\end{equation}
As for the initial data $u_0$, we assume
\begin{equation}
\label{hypotheses.u0}
\tag{$H_{u_0}$}
u_0\ge 0, \quad u_0\in L^\infty(\mathbb{R}),\quad u_0=0\text{ in }\H,\quad \int_{\R}u_0(x)(1+x^2)\,dx<\infty.
\end{equation}

It is easy to prove by means of a fixed point argument that there
exists a unique solution $u\in
C\big([0,\infty);L^1\big(\R,(1+x^2)\,dx\big)\big)$ to problem~\eqref{problem}; see~\cite{CEQW} for a similar reasoning.

\medskip

\noindent\emph{Remark. } The assumptions on the second momenta, both on the initial data and the solution, as well as the sign restriction and the boundedness of the initial data,  are not needed for the existence and uniqueness proof. However, they play a role in our asymptotic results. Nevertheless, the hypothesis  on the sign of the initial data can be easily removed \emph{a posteriori}, once we know the result for signed solutions, using the linearity of the equation; see below for the details.

\medskip

Evolution problems with this type of diffusion have been widely
considered in the literature, since they can be used to model the
dispersal of a species by taking into account long-range effects
\cite{BZ,CF,Fi}. It has also been  proposed as a
model for phase transitions \cite{BCh1,BCh2}, and, quite
recently, for image enhancement \cite{GO}.

The large time behavior for this kind of problems in large dimensions, $N\ge3$,  was studied in~\cite{CEQW}. In this case
\begin{equation*}\label{eq:result.large.dimensions}
t^{N/2}\|u(\cdot,t)-M^*\phi(\cdot)\Gamma_\a(\cdot,t)\|_{L^\infty(\R^N)}\to 0\quad\mbox{as }
t\to\infty,
\end{equation*}
where $\Gamma_\a(x,t)=(4\pi \a t)^{-N/2}\textrm{e}^{-|x|^2/(4\a t)}$ is the fundamental solution of the standard local heat equation with diffusivity $\a=\frac1{2}\int_\mathbb{R} J(z)|z|^2\,dz$, $\phi$ is the unique solution to
\begin{equation}\label{eq:stationary.bounded}
L\phi=0\quad \mbox{in }\R^N\setminus\H,\qquad
\phi=0\quad\mbox{in }\mathcal{H},
\end{equation}
such that $\phi(x)\to 1$ as $|x|\to\infty$,
and  $M^*=\int_{\mathbb{R}^N} u_0(x)\phi(x)\,dx$. The quantity $M^*>0$ turns out to be the asymptotic mass.

In spatial dimension $N=1$, problem
\eqref{eq:stationary.bounded} admits no bounded solution except $\phi=0$; see Proposition~\ref{lemma:uniqueness.stationary}.
Moreover, as we will see, the solution to \eqref{problem} loses asymptotically all its mass. However, there is a residual asymptotic first momentum. Thus, the asymptotic behavior is not expected to be given in terms of $\Gamma_\a$, a function that conserves mass, but in terms of the  \emph{dipole} solution to the heat equation with diffusivity $\a$,
$$
\mathcal{D}_\a(x,t)=\partial_x \Gamma_\a(x,t)=-\frac{x}{2\a t}\Gamma_q(x,t),
$$
which has $\delta'$, the derivative of the Dirac mass, as initial data,  and preserves the first momentum.

To explore what may be the large time behavior when $N=1$, we first consider the case in which the hole $\H$ contains a \emph{large} interval, with a diameter bigger than the radius of the support of the kernel; that is, we may take $2a_0>d$ in~\eqref{hypotheses.H}. In this situation the domain $\R\setminus\H$ has two disconnected components which can be treated independently as problems on a half line, maybe with some extra holes.  The problem on the half line
\begin{equation}
\label{eq:half.line}
\begin{cases}
\displaystyle \partial_t u(x,t)=Lu(x,t)\quad&x\ge0,\ t>0,\\
u(x,t)=0,&-d<x<0,\ t>0,\\
u(x,0)=u_0(x),&x>-d,
\end{cases}
\end{equation}
where $u_0(x)=0$ if $-d< x<0$, was tackled in~\cite{CEQW2}.  Under some assumptions on the initial data similar to~\eqref{hypotheses.u0}, the authors prove by means of a symmetrization argument that the solutions of~\eqref{eq:half.line} satisfy
\begin{equation*}
\label{eq:main.result.half.line}
\sup_{x\in\R_+}\frac {t^{3/2}}{x+1}\Big|u(x,t)+2M_1^*\frac{\phi(x)}{x}\mathcal{D}_\a(x,t)\Big|\to0\quad\text{as }t\to\infty,
\end{equation*}
where $\phi$ is the unique solution to
\begin{equation*}
\label{eq:stationary.problem.half.line}
L\phi=0\quad\text{in }\overline\R_+,\qquad \phi=0\quad\text{in }(-d,0),\qquad |\phi(x)-x|\le C<\infty, \quad x\in \R_+,
\end{equation*}
and $M_1^*=\int_0^\infty u_0(x)\phi(x)\,dx$. The analysis can be easily extended to the case in which the half line has holes, as long as they are bounded. Therefore, if $2a_0> d$, the large time behavior of solutions to~\eqref{problem} is given by
\begin{equation}\label{main result}
\sup_{x\in\R}\frac {t^{3/2}}{|x|+1}\left|u(x,t)+2
\frac{\phi_0(x)}{x}\mathcal{D}_\a(x,t)\right|
\to 0\mbox{ as }t\to\infty,
\end{equation}
where $\phi_0$ is the unique solution to the stationary problem
\begin{equation}\label{stationary}
L\phi=0\text{ in }\R\setminus\H,\quad
\phi=0\text{ in }\H,\quad |\phi(x)-\max\{b^+x,-b^-x\}|\le C<\infty, \quad x\in \R,
\end{equation}
with constants $b^\pm$ given by
$$
b^\pm=\int_{\R_\pm} u_0(x)\phi_\pm(x)\,dx,
$$
where the functions $\phi_\pm$ are the unique solutions to
\begin{equation}
\label{eq:phi.pm}
L\phi_\pm=0\text{ in }\R\setminus\H,\qquad
\phi_\pm=0\text{ in }\H,\qquad |\phi_\pm(x)-\max\{\pm x,0\}|\le C<\infty, \quad x\in \R.
\end{equation}

If the hole $\H$ does not contain a large interval, it does not disconnect the real line in independent components, and the symmetrization technique used in~\cite{CEQW2} cannot be applied. However, though a completely different approach is needed, the asymptotic result is still true. This is the main outcome of the present paper.
\begin{teo}\label{thm:main}
Let $\H$, $J$ and $u_0$ satisfy \eqref{hypotheses.H}, \eqref{hypotheses.J} and \eqref{hypotheses.u0} respectively. Let $\bar{M}_1^\pm=\int_{\R_\pm} u_0(x)\phi_\pm(x)\,dx$,
where $\phi_\pm$ satisfy \eqref{eq:phi.pm}. Let $\phi_0$ be the solution to~\eqref{stationary} with $b^\pm=\bar{M}_1^\pm$. Then,
\eqref{main result} holds.
\end{teo}

The sign restriction is now easily removed. Indeed, if $u^\pm$ are the solutions with initial data $\{u_0\}_\pm$, then, by the linearity of the equation,  $u=u^+-u^-$. Since
$$
\bar M_1^\pm=\int_{\R_\pm}\left(\{u_0(x)\}_+-\{u_0(x)\}_-\right)\phi_\pm(x)\,dx,
$$
the result for general data will follow from the results for $u^+$ and $u^-$. Notice, however, that in the case of initial data with sign changes it may happen that both $\bar M_1^\pm=0$.  In this situation our result is not optimal (solutions decay faster), and we should look for a different scaling.

\medskip

\noindent\emph{Remark. } If the hole contains two large intervals, the function $\phi_0$
is identically 0 in the interval in between. Therefore,  the scaling we are using is not adequate to characterize the asymptotic behavior there. Indeed, the decay rate for the solution of the Dirichlet problem  in a bounded set
 is exponential, and the asymptotic profile is an eigenfunction associated to the first eigenvalue of the operator $L$  with zero Dirichlet boundary conditions in the complement of the set~\cite{CCR}.

\medskip

It is worth noticing that the function giving the asymptotic behavior is  as smooth in the set $(\R\setminus\H)\times\R_+$ as $\phi_0$. This latter function is $C^2$ smooth under our assumptions \eqref{hypotheses.H} on $J$. This may look a bit surprising, since it is well-known that the solution to problem~\eqref{problem} is as smooth as the initial data $u_0$, but not more; see the representation formula~\eqref{eq:representation.formula} below.

Given $\xi_1,\xi_2>0$, there exist constants $C_1, C_2>0$ such that $C_1 \le t|\mathcal{D}_\a(x,t)|\le C_2$  in the \emph{outer region}  $\xi_1\le |x| /t^{1/2}\le \xi_2$.  Therefore, Theorem~\ref{thm:main} implies that in the \emph{far field} scale, $|x|\sim \xi t^{1/2}$, the solution satisfies $0<c_1\le t|u(x,t)|\le c_2<\infty$.
In the \emph{near field}, $0\le |x|\le t^{1/2}h(t)$, $\lim_{t\to\infty}h(t)=0$, the solution resembles
$\frac{\phi_0(x)}{2\a^{3/2}\sqrt\pi t^{3/2}}$ in the limit $t\to\infty$. Hence, there is a continuum of possible decay rates,  starting with the decay rate $O(t^{-1})$, holding in the far field, all the way up to $O(t^{-3/2})$, that takes place on compact sets. The rate depends on the scale, and is explicitly given by the relation $t^{3/2}/|x|$.

One of the first steps in the proof of Theorem~\ref{thm:main} is the obtention of the global decay rate, $\|u(\cdot,t)\|_{L^\infty(\R)}=O(t^{-1})$; see Section~\ref{sect:global.size.estimate}. This is done through a delicate iterative decay rate improvement, since no global supersolution with the right decay is available. In the \emph{very far field}, $|x|\ge t^{1/2}g(t)$, $\lim_{t\to\infty}g(t)=\infty$, Theorem~\ref{thm:main} does not give any further information. In fact the result there is trivial once we know the global decay rate. Nevertheless, we will be able to prove that $u(\cdot,t)=o(t^{-1})$ in this region; see Theorem~\ref{thm:very far field}.

Our results are in sharp contrast to what happens for the Cauchy problem, $\H=\emptyset$. Indeed,  for any dimension $N$,
\[
t^{N/2}\|u(\cdot,t)-v(\cdot,t)\|_{L^\infty(\R^N)}\to 0\quad\mbox{as }t\to\infty,
\]
where $v$ is the solution to the heat equation with diffusivity $\a=\frac1{2}\int_{\mathbb{R}^N} J(z)|z|^2\,dz$ and initial condition $v(\cdot,0)=u(\cdot,0)\in L^1(\R^N)\cap L^\infty(\R^N)$; see \cite{CCR,IR1}. Therefore,
\[
t^{N/2}\|u(\cdot,t)-M\Gamma_\a(\cdot,t)\|_{L^\infty(\R^N)}\to0\quad\mbox{as }t\to\infty, \qquad M=\int_{\mathbb{R}^N} u_0.
\]
Thus, there is no difference in the asymptotic behavior between small or large dimensions.
Notice that, though the global decay rates for the Cauchy problem and for the problem with holes coincide when $N\ge3$, they differ when $N=1$.

In the presence of holes, the case $N=2$ is borderline. Mass is expected to decay to zero with a logarithmic rate, $M(t)=O((\log t)^{-1})$. The asymptotic behavior in the far field will be given by the fundamental solution, but now with a variable mass, decaying to zero logarithmically with time. On the other hand, the stationary solution giving, after and adequate size scaling, the near field limit, behaves logarithmically at infinity.   Thus, logarithmic corrections  are required. This critical case will be considered elsewhere.

Let us finally mention that in the case of the standard (local) heat equation in spatial dimension $N=1$ the holes are always \lq\lq big''. Hence, the analysis of the asymptotic behavior can be reduced without exception to the case of a half line. A complete such analysis can be found in~\cite{CEQW2}.

\medskip

\noindent\textsc{Organization of the paper. } Section~\ref{sect:stationary.problem} is devoted to the study of the stationary problem~\eqref{stationary}. We will prove that it has a unique solution, and will obtain some bounds for its first and second derivatives that will be required to study the near field limit. As part of the proof of uniqueness, we find that there is no bounded solution to~\eqref{eq:stationary.bounded}. In Section~\ref{sect:conservation.law} we find a conservation law, we prove that the mass decays to zero and we find the asymptotic first momentum in terms of the initial condition.
Sections~\ref{sect:global.size.estimate} and~\ref{sect:refined.size.estimate} are devoted to obtain good size estimates. In Sections~\ref{sect:far.field} and~\ref{sect:near.field} we study  the asymptotic profile: the first one deals with the far field limit and the second with the near field limit. Finally,  in Section~\ref{sect:near.field}, we obtain an improved estimate for the decay rate in the very far field.

\section{The stationary problem}
\label{sect:stationary.problem}
\setcounter{equation}{0}

In this section we  prove the existence
of a unique  solution to~\eqref{stationary} for $b^\pm$ arbitrary nonnegative constants, and obtain some estimates for its derivatives that will be used to obtain the near field limit.

\subsection{Existence and uniqueness}
The function $\phi$ solving~\eqref{stationary} will be obtained as the limit when $n$ tends to infinity of the solutions $\phi_n$ to
\begin{equation*}\label{Pn}
L\phi_n=0\quad\text{in }B_n\setminus\H,\qquad \phi_n=0\quad\text{in }\H,\qquad \phi_n=\max\{b^+(x-a),-b^-(x+a)\} \quad\text{in }B_{n+d}\setminus B_n,
\end{equation*}
where $a$ is as in~\eqref{hypotheses.H}, and $n>a$.
Existence and uniqueness for such problem is a consequence of \cite[Lemma 3.1]{CEQW}

The existence of a limit $\phi=\lim_{n\to\infty}\phi_n$ which satisfies~\eqref{stationary} will follow from the fact that all the functions $\phi_n$ are trapped between
$$
\underline{S}(x)=\max\{b^+(x-a),-b^-(x+a),0\},
$$
which is trivially a subsolution in the whole real line, since it is a maximum of solutions, and the function
\begin{equation}
\label{eq:super}
\overline{S}(x)=k+\max\{b^+x,-b^-x\},\quad x\in\R\setminus\H, \qquad \overline{S}(x)=0, \quad x\in\H,
\end{equation}
which is shown next to be a supersolution in $\R\setminus\H$ if $k$ is large enough.

\begin{lema} Let $\H$ and $J$ satisfy respectively hypotheses~\eqref{hypotheses.H} and \eqref{hypotheses.J}.
If $k>0$ is large enough, then the function $\overline S$ defined in \eqref{eq:super} satisfies
$L\overline S\le0$ in $\R\setminus\H$.
\end{lema}
\begin{proof}
We assume that $x\ge0$, $x\not\in\H$, which implies in particular that $x\ge a_0$. The case $x\le 0$, $x\not\in\H$ is treated in a similar way.  We have  two possibilities.

\noindent (i) If $x-d\ge-a_0$, taking $k\ge b^+a_0$ we get
$$
L\overline S(x)=\int_{\max\{x-d,a_0\}}^{x+d}J(x-y)(b^+y+k)\,dy- (b^+x+k)=-\int_{x-d}^{\max\{x-d,a_0\}}(b^+y+k)\le 0.
$$

\noindent (ii) If $x-d<-a_0$, which is only possible if $d\ge 2a_0$, then
\[
\begin{aligned}
L\overline S(x)=& \int_{x-d}^{-a_0} J(x-y)(-b^-y-k)\,dy+ \int_{a_0}^{x+d} J(x-y)(b^+y+k)\,dy- (b^+x+k)\\
=& -b^-\int_{x-d}^{-a_0}J(x-y)y\,dy-b^+\int_{x-d}^{a_0}J(x-y)y\,dy-k\int_{-a_0}^{a_0}J(x-y)\,dy.
\end{aligned}
\]
Since $J$ is nonincreasing in $\R_+$  and  $d-x-a_0\ge 0$,  then
\begin{equation}
\label{eq:nu}
\int_{-a_0}^{a_0}J(x-y)\,dy=\int_{x-a_0}^{x+a_0}J(y)\,dy\ge\int_{x-a_0}^{x+a_0}J(y+d-x-a_0)\,dy=  \int_{d-2a_0}^dJ(y)\,dy>0.
\end{equation}
Moreover, since in this case $d\ge2a_0$, we have
$$
\left|b^-\int_{x-d}^{-a_0}J(x-y)y\,dy+b^+\int_{x-d}^{a_0}J(x-y)y\,dy\right|\le (d-a_0)(b^-+b^+).
$$
Therefore, any $k\ge (d-a_0)(b^-+b^+)/\int_{d-2a_0}^dJ(y)\,dy$ will do the job.
\end{proof}

\begin{prop}
\label{prop:existence.stationary}
Let $\H$ and $J$ satisfy respectively hypotheses~\eqref{hypotheses.H} and \eqref{hypotheses.J}. Then there exists a solution to~\eqref{stationary}.
\end{prop}

\begin{proof}
By comparison, $\underline S(x)\le \phi_n(x)\le \overline S(x)$ for all $n\ge a$ and $x\in B_{n+d}$. This implies  in particular that $\phi_n\le\phi_{n+1}$ in the annular region $B_{n+d}\setminus B_n$, and hence, again by comparison, in the whole ball $B_{n+d}$. We conclude that the monotone limit
$$
\phi(x)=\lim_{n\to\infty}\phi_n(x)
$$
exists and is finite.
It is then trivially checked that $\phi$ solves \eqref{stationary}.
\end{proof}

Let $\phi_1,\phi_2$ be solutions to \eqref{stationary}.  Then, $\phi=\phi_1-\phi_2$ is  a bounded solution of~\eqref{eq:stationary.bounded}. If $J>0$ in $(-d,d)$, uniqueness for problem~\eqref{stationary}  then follows from the following lemma.

\begin{prop}
\label{lemma:uniqueness.stationary}
Under the assumptions of Proposition~\ref{prop:existence.stationary}, if $J>0$ in $B_d$, the unique bounded solution to~\eqref{eq:stationary.bounded} is $\phi=0$.
\end{prop}
\begin{proof}
The function $\phi_\varepsilon=\phi-\varepsilon \overline S$ satisfies $L\phi_\varepsilon\ge0$ in $\R\setminus\H$, and reaches its maximum at some finite point $\bar x$, since by construction $\phi_\varepsilon(x)\to-\infty$ as $|x|\to\infty$. A standard (for nonlocal operators) argument shows that if $\phi_\varepsilon(\bar x) > 0$ we reach a contradiction. Indeed, if $\phi_\varepsilon(\bar x) > 0$,  we deduce that $\phi_\varepsilon$ is constant in $(\bar x-d, \bar x+d)$. We can thus propagate the maximum  to the whole connected component of $\R\setminus \H$ where $\bar x$ lies, which leads to a contradiction for points near the boundary of this component.  Then, passing to the limit as $\varepsilon\to0$,
we obtain $\phi\le0$. The same argument applied to $-\phi$ leads to  $\phi \ge 0$.
\end{proof}

\subsection{Estimates for the derivatives}
In the course of the study of the near field limit we will need estimates for some derivatives of $\psi(x)=\phi(x)-\max\{b^+x,-b^-x\}$. They will be obtained here.
The proofs of these estimates use that $\psi$ solves a problem of the form
\begin{equation}
\label{eq:inhomogeneous}
\partial_t u-L u=f\quad\text{in }\R\times\R_+, \qquad u(x,0)=u_0(x),\quad x\in\R.
\end{equation}
By the variation of constants formula, solutions to~\eqref{eq:inhomogeneous} can be written in terms of the fundamental solution $F=F(x,t)$ for the operator $\partial_t-L$ in the whole space, which
can be decomposed as
\begin{equation}\label{fund-sol}
F(x,t)=\textrm{e}^{-t}\delta(x)+W(x,t),
\end{equation}
where $\delta(x)$ is the Dirac mass at the origin and $W$ is a nonnegative smooth function defined via its Fourier transform,
\begin{equation*}
\label{eq:transform.omega}
\widehat W(\xi,t)=\textrm{e}^{-t}\left(\textrm{e}^{\hat J(\xi)t}-1\right);
\end{equation*}
see~\cite{CCR}. Thus,
\begin{equation}
\label{eq:representation.formula}
\begin{aligned}
u(x,t)=&\textrm{e}^{-t}u_0(x)+\int_{\mathbb{R}} W(x-y,t)u_0(y)\,dy\\
&+\int_{0}^t\textrm{e}^{-(t-s)}f(x,s)\,ds+\int_{0}^t\int_{\mathbb{R}} W(x-y,t-s)f(y,s)\,dyds.
\end{aligned}
\end{equation}
Therefore, estimates for solutions to~\eqref{eq:inhomogeneous}, and in particular for $\psi$, will follow if we have good estimates for the right hand side of the equation,  $f$,  and  for the regular part, $W$, of the fundamental solution.

The asymptotic convergence
of $W$ to the fundamental solution of the local heat equation
with diffusivity $\a$ yields a first class of estimates. Indeed,  for all $s\in\mathbb{N}$,
\begin{equation}\label{estima-W}
\|\partial_x^s
W(\cdot,t)-\partial_x^s\Gamma_\a(\cdot,t)\|_{L^\infty(\mathbb{R})}\le Ct^{-\frac{s+2}2};
\end{equation}
see \cite{IR1}. Hence,  in
particular,
\begin{equation}\label{decaimiento-W}
|\partial_x^s W(x,t)|\le Ct^{-\frac{s+1}2}\quad\text{for all }s\in\mathbb{N}.
\end{equation}
These estimates give the right order of time decay, and will prove to be useful later, in Section~\ref{sect:near.field}. However,  they
do not take into account the spatial structure of $W$, and are not enough for our present goal.
Instead, we will use that
\begin{equation}\label{eq:estimates.W}
|\partial_x^s W(x,t)|\le C\frac t{|x|^{3+s}},
\qquad
\int_{\R}|\partial_x^s W(x,t)|\,dx\le C t^{-s/2}\qquad \text{for all }s\in\mathbb{N}.
\end{equation}
These estimates were proved in~\cite{TW} through a comparison argument, using that $W$ is a solution to
\begin{equation}\label{eq-W}
\begin{cases}
\partial_t W(x,t)-LW(x,t)=e^{-t}J(x)\quad&\mbox{in }\mathbb{R}\times\R_+,\\
W(x,0)=0\quad&\mbox{in }\mathbb{R}.
\end{cases}
\end{equation}

\begin{lema}
\label{lemma:estimates.derivatives.phi}
Assume the hypotheses of Proposition~\ref{prop:existence.stationary}. Let $\phi$ satisfy~\eqref{stationary} and $\psi(x)=\phi(x)-\max\{b^+x,-b^-x\}$.
There exists a constant $C>0$
such that
\begin{equation}\label{psi}
|\psi'(x)|\le \frac C{|x|^{4/5}}, \qquad |\psi''(x)|\le  \frac C{|x|^{5/3}}, \qquad x\in\R\setminus\H.
\end{equation}
\end{lema}
\begin{proof} \textit{Representation formula and scaling. } The function $\psi$ is a solution to~\eqref{eq:inhomogeneous}
with right hand side
$$
f=-\X_\H(J*u)+\X_{B_d}Lh, \qquad h(x)=\max\{b^+x,-b^-x\},
$$
and initial data $u_0(x)=\psi(x)$. Hence, the variations of constants formula yields
$$
\psi(x)=\frac1{1-\textrm{e}^{-t}}\int_\R W(x-y,t)\psi(y)\,dy+\frac1{1-\textrm{e}^{-t}}
\int_0^t\int_{-\max\{a,d\}}^{\max\{a,d\}}
W(x-y,t-s)f(y)\,dy\,ds.
$$
Now, for $k>0$, let $\psi^k(x)=k^{-\alpha}\psi(kx)$. Then,
$$
\psi^k(x)=\frac{k^{-\alpha}}{1-\textrm{e}^{-t}}\int_\R W(kx-y,t)\psi(y)\,dy+
\frac{k^{-\alpha}}{1-\textrm{e}^{-t}}\int_0^t\int_{-\max\{a,d\}}^{\max\{a,d\}} W(kx-y,t-s)f(y)\,dy\,ds.
$$

\medskip

\noindent\textit{Estimate for the first derivative. } We thus have
\begin{equation*}\label{psi-prima}
\begin{aligned}
(\psi^k)'(x)=&\underbrace{\frac{k^{1-\alpha}}{1-\textrm{e}^{-t}}\int_\R \partial_x W(kx-y,t)\psi(y)\,dy}_{\mathcal A}\\
&\qquad+
\underbrace{\frac{k^{1-\alpha}}{1-\textrm{e}^{-t}}\int_0^t\int_{-\max\{a,d\}}^{\max\{a,d\}}\partial_x W(kx-y,t-s)f(y)\,dy\,ds}_{\mathcal B}.
\end{aligned}
\end{equation*}
In order to bound ${\mathcal A}$ we use that $\psi$ is bounded in $\R$, together with the second estimate in~\eqref{eq:estimates.W} with $s=1$, to obtain
$$
|{\mathcal A}|\le C\frac{k^{1-\alpha}}{1-\textrm{e}^{-t}}\int_\R |\partial_x W(y,t)|\,dy\le C\frac{k^{1-\alpha}}{1-\textrm{e}^{-t}}\,t^{-1/2}\le C \qquad\text{for } t=k^{2(1-\alpha)}.
$$
On the other hand, since $f$ is bounded, for $t=k^{2(1-\alpha)}$, $|x|=1$ and $k\ge2\max\{a,d\}$ we get, using the first estimate in~\eqref{eq:estimates.W} with $s=1$,
\[\begin{aligned}
|{\mathcal B}|&
\le C\frac{k^{1-\alpha}}{1-\textrm{e}^{-t}}\int_0^t\int_{-\max\{a,d\}}^{\max\{a,d\}}\frac{t-s}{|kx-y|^4}\,dy\,ds
\le C\frac{k^{1-\alpha}}{1-\textrm{e}^{-t}}\frac{t^2}{k^4|x|^4}\\&=
C\frac1{1-\textrm{e}^{-k^{2(1-\alpha)}}}\,
k^{1-\alpha-4+4(1-\alpha)}\le C \qquad\text{for all }\alpha\ge1/5.
\end{aligned}
\]

Summarizing, if we take $\alpha=1/5$ in the definition of $\psi^k$, we have
\[
k^{4/5}|\psi'(kx)|=|(\psi^k)'(x)|\le C, \qquad |x|=1, \ k\ge2\max\{a,d\},
\]
which is immediately translated into
\[
|\psi'(x)|\le \frac C{|x|^{4/5}},\qquad |x|\ge2\max\{a,d\}.
\]
This proves the first estimate in~\eqref{psi} except in a bounded set. However, $\psi$ is smooth in $\R\setminus\H$, and hence the estimate is true everywhere outside $\H$.

\medskip

\noindent\textit{Estimate for the second derivative. } We have
\[
\begin{aligned}
(\psi^k)''(x)=&\underbrace{\frac{k^{2-\alpha}}{1-\textrm{e}^{-t}}\int_\R \partial_x^2W(kx-y,t)\psi(y)\,dy}_{\mathcal A}\\
&\qquad+
\underbrace{\frac{k^{2-\alpha}}{1-\textrm{e}^{-t}}\int_0^t\int_{-\max\{a,d\}}^{\max\{a,d\}} \partial_x^2W(kx-y,t-s)f(y)\,dy\,ds}_{\mathcal B}.
\end{aligned}
\]
Now, since $\psi$ is bounded, using the second estimate in~\eqref{eq:estimates.W} with $s=2$ we get
$$
|\mathcal A|
\le C\frac{k^{2-\alpha}}{1-\textrm{e}^{-t}}\int_\R |\partial_x^2W(y,t)|\,dy\le C\frac{k^{2-\alpha}}{1-\textrm{e}^{-t}}t^{-1}\le C \qquad\text{for } t=k^{2-\alpha}.
$$
Then, since $f$ is bounded, for $t=k^{2-\alpha}$, $|x|=1$ and $k\ge2\max\{a,d\}$ we get, using the first estimate in~\eqref{eq:estimates.W} with $s=2$,
\[\begin{aligned}
|\mathcal B|&
\le C\frac{k^{2-\alpha}}{1-\textrm{e}^{-t}}\int_0^t\int_{-\max\{a,d\}}^{\max\{a,d\}} \frac{t-s}{|kx-y|^5}\,dy\,ds\le C\frac{k^{2-\alpha}}{1-\textrm{e}^{-t}}\frac{t^2}{k^5|x|^5}\\
&=
C\frac1{1-\textrm{e}^{-k^{2-\alpha}}}\,
k^{2-\alpha-5+2(2-\alpha)}\le C\qquad\text{for all }\alpha\ge1/3.
\end{aligned}
\]

Summarizing, taking $\alpha=1/3$ in the definition of $\psi^k$, we have
\[
k^{5/3}|\psi''(kx)|=|(\psi^k)''(x)|\le C,  \qquad |x|=1, \ k\ge2\max\{a,d\},
\]
which immediately yields the second estimate in~\eqref{psi}, once we notice that $\psi$ is smooth outside the hole $\H$.
\end{proof}

\section{Conservation law, mass decay and asymptotic first momentum}
\label{sect:conservation.law}
\setcounter{equation}{0}

Comparison with the solution $u_c$ to the Cauchy problem with initial data $u_0$ gives a first estimate for the decay of the solution to~\eqref{problem}, since we know that $\|u_c(\cdot,t)\|_{L^\infty(\R)}=O(t^{-1/2})$; see for instance \cite{CCR,IR1}. However, this decay rate is not optimal; see Section~\ref{sect:global.size.estimate}. The idea to improve it is to use the following inequality, that comes from  comparison with the solution of the Cauchy problem with initial datum $u(x,\bar t)$, combined with estimate \eqref{decaimiento-W} with $s=0$ for $W$,
\begin{equation}\label{estimate-rossi}
\begin{array}{rcl}
u(x,t)&\le& \displaystyle e^{-(t-\bar t)}u(x,\bar t)+\int_{\mathbb{R}} W(x-y,t-\bar t)u(y,\bar t)\,dy\\
&\le& \displaystyle
e^{-(t-\bar t)}\|u_0\|_{L^\infty(\mathbb{R})}+\bar C(t-\bar t)^{-1/2}M(\bar t).
\end{array}
\end{equation}
If we were able to control the mass at $\bar t$ in terms of the size of $u$ at that time, which is estimated by the decay rate of $u$ available at this moment, we would get, taking $\bar t=t/2$,  a better decay rate for $u$; see the next section for the details.
Hence, we need to control the mass  in terms of the size of $u$. This is the first aim of this section. As a by-product we get the convergence of the first momenta of the solution in $\mathbb{R}_\pm$ towards non-trivial asymptotic values which can be computed in terms of the initial data. Finally, we obtain an estimate for the second momentum which plays a role in the
proofs of the far  limit.

In order to get the required results we  need a conservation law.
\begin{prop} \label{conservation} Assume the hypotheses of Theorem~\ref{thm:main}. Let $u$ be the solution to \eqref{problem}. Let $\phi$ be such that $L\phi=0$ in $\R\setminus\H$, $\phi=0$ in $\H$ and $0\le \phi(x)\le D(1+|x|)$ for a certain constant $D$.  Then, for every $t>0$,
\begin{equation*}\label{cl}
M_\phi(t):=\int_{\mathbb{R}} u(x,t)\phi(x)\,dx=\int_{\mathbb{R}}
u_0(x)\phi(x)\,dx.
\end{equation*}
\end{prop}
\begin{proof}
Since
$u\in C\big([0,\infty);L^1\big(\R,(1+|x|)\,dx\big)\big)$, the growth condition on $\phi$ implies
$M_\phi(t)<\infty$.
In addition, using the equation in~\eqref{problem}, we get
$\int_\R
|\partial_t u(x,t)|\phi(x)\,dx<\infty$. Therefore, we may differentiate under the integral sign to obtain, after applying
Tonelli's Theorem,
$$
M_\phi'(t)=\int_\R \partial_t u(x,t)\phi(x)\,dx=\int_\R Lu(x,t)\phi(x)\,dx=\int_\R u(x,t)L\phi(x)\,dx=0.
$$
\end{proof}
Now we can relate the mass decay rate to the decay rate of the solution.
\begin{prop}\label{mass-decay-rate} Under the assumptions of Theorem~\ref{thm:main}, there exists a constant $K_1$ such that
\begin{equation}\label{mass-estimate}
M(t)\le K_1\|u(\cdot,t)\|_{L^\infty(\R)}^{1/2}\quad\mbox{for every }t\ge0.
\end{equation}
\end{prop}
\begin{proof} Take $\phi_1$ the solution to \eqref{stationary} with $b^\pm=1$, and $\bar M_1=\int_\R u_0(x)\phi_1(x)\,dx$. Then, we take  $\sigma$ large so that on the one hand $|x|\le \sigma\phi_1(x)$ if $|x|\ge a$, and on the other hand $\frac{\sigma \bar M_1}{2\|u_0\|_{L^\infty(\R)}}\ge a^2$.

 Let $\delta(t)>a$, $t\ge 0$, to be chosen later. We have
\begin{equation*}\label{mass-eq}
\begin{aligned}
\int_{\mathbb{R}} u(x,t)\,dx&=\int_{|x|<\delta(t)}u(x,t)\,dx+\int_{|x|>\delta(t)} u(x,t)\,dx\\
&\le \int_{|x|<\delta(t)} \|u(\cdot,t)\|_{L^\infty(\R)}\,dx+\frac1{\delta(t)}\int_{|x|>\delta(t)}u(x,t)|x|\,dx\\
&\le 2\delta(t)\|u(\cdot,t)\|_{L^\infty(\R)}+ \frac\sigma{\delta(t)}\int_{\mathbb{R}} u(x,t)\phi_1(x)\,dx\\
&=2\delta(t)\|u(\cdot,t)\|_{L^\infty(\R)}+ \frac{\sigma \bar M_1}{\delta(t)}.
\end{aligned}
\end{equation*}
The choice $\delta(t)=\big(\frac{\sigma \bar M_1}{2 \|u(\cdot,t)\|_{L^\infty(\R)}}\big)^{1/2} $ optimizes the right hand side of this estimate, and yields the desired result with $K_1=2(2\sigma \bar M_1)^{1/2}$. Notice that $\delta(t)$ is a nondecreasing function of time. Hence, $\delta(t)\ge \delta(0)=\big(\frac{\sigma \bar M_1}{2 \|u_0\|_{L^\infty(\R)}}\big)^{1/2}\ge a$, as required.
\end{proof}

We also have the following result regarding the first momenta for $x>0$ and $x<0$.
\begin{prop}\label{Mlimits} Assume the hypotheses of Theorem~\ref{thm:main}. Let $\phi_\pm$ given by~\eqref{eq:phi.pm},
$M_1^\pm(t)=\int_{\R_\pm} u(x,t)|x|\,dx$, and $\bar{M}_1^\pm(t)=\int_{\R} u_0(x)\phi_\pm(x)\,dx$.
Then
\begin{equation*}\label{momentum-order}
|M_1^\pm(t)-\bar M_1^\pm|\le CM(t)\le C\|u(\cdot,t)\|_{L^\infty(\R)}^{1/2}.
\end{equation*}
\end{prop}
\begin{proof} Since $|\phi_+(x)-\max\{x,0\}|\le C$, by using \eqref{cl} we get
$$
\big|M_1^+(t)-\bar M_1^+\big|
\le\int_{0}^{\infty}
u(x,t)\left|x-\phi_+(x)\right|\,dx+\int_{-\infty}^0 u(x,t)\phi_+(x)\,dx
\le CM(t).
$$
A similar analysis gives the statement concerning $M_1^-(t)$ and $\bar M_1^-$.
\end{proof}

\noindent\emph{Remark. } Since the solution $u$ decays to 0, this implies in particular that $\bar M_1^\pm$ are the asymptotic left and right first momenta.

\medskip

\begin{coro}
Assume the hypotheses of Theorem~\ref{thm:main}.
Let $M_1(t)=\int_\R u(x,t)|x|\,dx$ and $\bar M_1=\int_\R u_0(x)\phi_1(x)$, where $\phi_1$ is the solution to~\eqref{stationary} with $b^\pm=1$. Then,
$$
|M_1(t)-\bar M_1|\le C M(t)\le C\|u(\cdot,t)\|^{1/2}_{L^\infty(\R)}.
$$
\end{coro}

\noindent\emph{Remark. } The uniform convergence of the solution to 0 implies then that $\bar M_1<\infty$ is the asymptotic first momentum. Hence, $M_1\in L^\infty(\R_+)$.

\medskip

Finally,  we control the growth rate of the second momentum in terms of the decay of the solution.
\begin{prop}
\label{prop:growth.rate.second.momentum}
Under the hypotheses of Theorem~\ref{thm:main}, there exists a constant $K_2$ such that
\begin{equation}\label{eq-M2prime}
\frac d{dt}M_2(t)\le cM(t)\le K_2\|u(\cdot,t)\|_{L^\infty(\R)}^{1/2}.
\end{equation}
\end{prop}
\begin{proof} Since $Lx^2=c$ and $\int_\H Lu(x,t)x^2\,dx=\int_\H x^2\int_\R J(x-y)u(y,t)\,dy\,dx\ge0$, applying Tonelli's Theorem and the symmetry of the kernel,
$$
M_2'(t)=\int_{\R\setminus\H} Lu(x,t) x^2\,dx\le \int_{\R} Lu(x,t) x^2\,dx=
\int_{\mathbb{R}} u(x,t)\,Lx^2\,dx=  c\int_{\mathbb{R}} u(x,t)\,dx.
$$
\end{proof}

\section{A global size estimate}
\label{sect:global.size.estimate}
\setcounter{equation}{0}

The aim of this section is to obtain a global size estimate for the solutions of~\eqref{problem}. In a later section we will see that it turns out to be optimal

\begin{teo}
\label{order}
Under the assumptions of Theorem~\ref{thm:main},   $\|u(\cdot,t)\|_{L^\infty(\R)}=O(t^{-1})$.
\end{teo}

The result is a corollary of the following lemma, which is obtained through an iterative procedure.

\begin{lema}
\label{lemma:iterative.estimate}
Assume the hypotheses of Theorem~\ref{thm:main}.
Let $\alpha_k=1-2^{-(k+1)}$, $t_k=2^{k-1}$, $k\in\mathbb{N}$.
There is a non-decreasing bounded sequence $\{C_k\}_{k=0}^\infty$ with $C_0\ge1$  such that
\begin{equation}
\label{eq:iterative.estimate}
\|u(\cdot,t)\|_{L^\infty(\R)}\le C_k t^{-\alpha_k},\quad t\ge t_k.
\end{equation}
\end{lema}

\noindent Indeed, once we prove the lemma, for $2^{k-1}<t\le 2^{k}$, $k\in\mathbb{N}$, we  have
$$
u(x,t)\le C_k t^{-1}\, t^{1-\alpha_k}\le C_k t^{-1} \, 2^{k/{2^{k+1}}}\le C t^{-1}
$$
for some constant $C$ independent of $k$, from where Theorem~\ref{order} follows immediately.

\medskip

\noindent\emph{Proof of Lemma~\ref{lemma:iterative.estimate}. }  The proof proceeds by induction.
Comparison with the solution of the Cauchy problem with the same initial data as $u$ shows that formula~\eqref{eq:iterative.estimate} holds for $k=0$ and some constant $C_0\ge1$. So we have to prove that if  the result is true up to a certain integer $k$, then it also holds for $k+1$.

If~formula~\eqref{eq:iterative.estimate} holds up to $k$,  estimate~\eqref{mass-estimate}  implies
$$
M\left(\frac t2\right)\le K_1C_k^{1/2}\left(\frac t2\right)^{-\alpha_k/2},\quad t\ge 2t_k=t_{k+1}.
$$
We now take $c$ such that $\textrm{e}^{-t/2}\|u_0\|_{L^\infty(\mathbb{R})}\le c t^{-1}$ for every $t>1$. Since   $C_k\ge1$, by the induction hypothesis, and $\alpha_k<1$, estimate~\eqref{estimate-rossi} with $\bar t=t/2$  yields
$$
\begin{aligned}
u(x,t)&\le \underbrace{\big(c+2\bar CK_1\big)}_{H}C_k^{1/2}t^{-\frac{1+\alpha_k}2}=
C_{k+1}t^{-\alpha_{k+1}}, \quad t\ge t_{k+1},
\end{aligned}
$$
if we take $C_{k+1}=HC_k^{1/2}$.

We may assume that $H\ge C_0^{1/2}$. Hence, the sequence has the required monotonicity,
$$
\frac{C_k}{C_{k+1}}=\Big(\frac{C_{k-1}}{C_k}\Big)^{1/2}=\cdots=\Big(\frac{C_0}{C_1}\Big)^{1/2^k}
=\Big(\frac{C_0^{1/2}}{H}\Big)^{1/2^k}\le 1.
$$
Moreover,
$$
C_k=H^{\sum_{n=0}^{k-1}\frac1{2^n}} C_0^{1/2^k}\le H^2 C_0<\infty.
$$
\qed

\medskip

The decay rate provided by~Theorem~\ref{order} combined with Propositions~\ref{mass-decay-rate}, \ref{Mlimits} and~\ref{prop:growth.rate.second.momentum} gives us better estimates for the mass and the  first two momenta.
\begin{coro}
\label{corollary:estimates.momenta}
Under the assumptions of Theorem~\ref{thm:main},
$$
\begin{array}{ll}
M(t)=O(t^{-1/2}),\qquad& |M_1^\pm(t)-\bar M_1^\pm|=O(t^{-1/2}),\\[8pt]
|M_1(t)-\bar M_1|=O(t^{-1/2}),\qquad&M_2(t)\le M_2(0)+O( t^{1/2}).
\end{array}
$$
\end{coro}

\section{A refined size estimate}
\label{sect:refined.size.estimate}
\setcounter{equation}{0}

Unfortunately, the global size estimate obtained in the previous section is to crude for our purposes. In order to prove our asymptotic results, we will need to combine it with a refined bound which gives the right decay of $u$ in all the scales up to the beginning of the far field scale, $|x|/t^{1/2}\le \xi^*$. The aim of this section is to obtain this refined size estimate.

The pursued bound will follow from comparison with
$$
V(x,t)= \begin{cases}C \textrm{e}^{\frac{-(|x|+b)^2}{4 \alpha t}} \frac{(|x|+b)}{t^{3/2}},\quad& x\in\R\setminus(-a_0,a_0),\\ 0,&  x\in(-a_0,a_0),
\end{cases}
$$
in the set $(\R\setminus\H)\cap {\mathcal{A}}_{\alpha,b,T}$, where
$$
{\mathcal{A}}_{\alpha,b,T}=\{(x,t): (|x|+b)^2\le 4\alpha t,\ t>T\},
$$
for suitable choices of positive constants $C$, $b$, $\alpha$ and $T$.  Notice that $V$ may be written in terms of the functions
$$
V_{\pm} (x,t)=C \textrm{e}^{\frac{-(\pm x+b)^2}{4 \alpha t}} \frac{(\pm x+b)}{t^{3/2}},
$$
which are both solutions to the (local) heat equation with diffusivity $\alpha$, as
$$
V(x,t)=\begin{cases}
V_+(x,t),\quad&x\ge a_0,\\
0,& x\in(-a_0,a_0),\\
V_-(x,t),&x\le -a_0.
\end{cases}
$$
We start by proving that both $V_+$ and $V_-$ are supersolutions to the nonlocal heat equation, the first one in ${\mathcal{A}}_{\alpha,b,T}\cap\{x\ge a_0\}$, and the second one in ${\mathcal{A}}_{\alpha,b,T}\cap\{x\le -a_0\}$,
for suitable choices of the parameters.

\begin{lema}
\label{lem:V+.supersolution}
Assume~\eqref{hypotheses.J}. There exist values $\alpha$, $b$ and $T$ such that
$$
\partial_t V_+  - LV_+\ge0\quad\mbox{in }{\mathcal{A}}_{\alpha,b,T}\cap\{x\ge a_0\},\qquad \partial_t V_-  - LV_-\ge0\quad\mbox{in }{\mathcal{A}}_{\alpha,b,T}\cap\{x\le-a_0\}.
$$
\end{lema}
\begin{proof}
We consider the statement for $V_+$. The one for $V_-$ is proved similarly. Thus, we restrict ourselves to  ${\mathcal{A}}_{\alpha,b,T}\cap\{x\ge a_0\}$.

A trivial computation shows that
$$
\partial_x^2V_+(y,t) = \frac{C}\alpha \textrm{e}^{\frac{-(y+b)^2}{4 \alpha t}} \frac{(y+b)}{t^{5/2}}\Big(\frac{(y+b)^2}{4 \alpha t}-\frac32\Big).
$$
Thus, if $(x,t)\in{\mathcal{A}}_{\alpha,b,T}\cap\{x\ge a_0\}$ and $|y-x|\le d$, and we take $b\ge 5d$,
$$
\frac{(y+b)^2}{4 \alpha t}\le\frac{(d+x+b)^2}{(x+b)^2}\frac{(x+b)^2}{4 \alpha t}\le \left(1+\frac{d}{x+b}\right)^2\le\left(1+\frac{d}{b}\right)^2<\frac{36}{25};
$$
hence, $\partial_x^2V_+(y,t)<0$ under these assumptions.

On the other hand,  using Taylor's expansion and the symmetry of $J$, we see that
$$
LV_+(x,t)=\int_{\mathbb{R}} J(x-y)\big(V_+(y,t)
-V_+(x,t)\big)\,dy  = \partial^2_x V_+(\bar y,t)\int_{\mathbb{R}}  J(x-y)\,\frac{(x-y)^2}{2}dy=\a\partial^2_x V_+(\bar y,t)
$$
for some $\bar y\in[x-d,x+d]$ that depends on $(x,t)$. Therefore,
$$
\begin{aligned}
LV_+(x,t)&\leq{\frac\a\alpha}\  C \textrm{e}^{\frac{-(x+b+d)^2}{4 \alpha t}} \frac{(x+b-d)}{t^{5/2}}\left(\frac{(x+b+d)^2}{4 \alpha t}-\frac32\right)\\
&\leq {\frac\a\alpha} \partial^2_x V_+(x,t)\ \textrm{e}^{\frac{-2(x+b)d -d^2}{4 \alpha t}}\; \left(\frac{x+b-d}{x+b}\right)\left(\frac{\frac{(x+b+d)^2}{4 \alpha t}-\frac32}{\frac{(x+b)^2}{4 \alpha t}-\frac32}\right)\\
&\leq {\frac\a{\alpha}^2}\partial_t V_+(x,t)\ \textrm{e}^{\frac{-d}{\sqrt{\alpha t}} -\frac{d^2}{4 \alpha t}}\frac45 \,\left(1-\frac {2d}{\sqrt{\alpha t}}-\frac{d^2}{2\alpha t}\right),
\end{aligned}
$$
where we have used that $V_+$ is a solution to the local heat equation with diffusivity $\a$. The desired result follows if we take $T$ big enough, so that $\frac d{\sqrt{\alpha t}}+\frac{d^2}{4\alpha t}\le \frac12$ for $t\ge T$, and choose $\alpha=\sqrt{\frac{2\a}{5\textrm{e}^{1/2}}}$.
\end{proof}

Since $L$ is a nonlocal operator, it is not true in general that $LV_+=LV$ in ${\mathcal{A}}_{\alpha,b,T}\cap\{x\ge a_0\}$, neither $LV_-=LV$ in ${\mathcal{A}}_{\alpha,b,T}\cap\{x\le -a_0\}$. Hence, in order to prove that $V$ is a supersolution we have to work a bit more.
\begin{lema}\label{super}
Assume~\eqref{hypotheses.J}. There exist constants $\alpha$, $b$, and $T>0$ such that,
$\partial_t V-L V\ge0$ in ${\mathcal{A}}_{\alpha,b,T}\cap\{|x|\ge a_0\}$.
\end{lema}
\begin{proof} We will prove the result in the region ${\mathcal{A}}_{\alpha,b,T}\cap\{x\ge a_0\}$. The result for $\mathcal{A}\cap\{x\le- a_0\}$ is obtained in a similar way.

We take $\alpha$, $b$ and $T>0$ as in Lemma~\ref{lem:V+.supersolution}. Since $x\ge a_0$,
$$
\partial_t V(x,t) = \partial_t V_+(x,t)\geq LV_+(x,t) = L V(x,t)  +\underbrace{\int_{x-d}^{x+d} J(x-y) (V_+(y,t)-V(y,t))\, dy}_{\mathcal{J}}.
$$
If $x-d\ge a_0$, then  $\mathcal{J}=0$, and we are done. On the other hand, if  $-a_0  \leq  x-d \leq a_0$,
$\mathcal{J}= \int_{x-d}^{a_0} J(x-y) V_+(y,t) dy\ge 0$, as desired.

We are only left with the case
$a_0-d\le x-d <-a_0$, which is only possible if $d>2a_0$.
In this situation,
$$
\mathcal{J}=\int_{-a_0}^{a_0} J(x-y) V_+(y,t)\, dy  + \int_{x-d}^{-a_0} J(x-y) ( V_+(y,t) -V_- ( y,t))\, dy.
$$
To proceed, we need to control the relative sizes of $V_+(y,t)$ and $V_-(y,t)$ for $y\in (x-d,-a_0)$. Notice that for such values of $y$ we have $a_0\le |y|\le d-a_0$ and $V_\pm(y,t)>0$.

On the one hand,
$$
 \frac{V_-(y,t)}{V_+(y,t)}=\textrm{e}^{\frac{-b|y|}{\alpha t}}\left(1+\frac{2|y|}{b-|y|}\right)
 \le 1+\frac{2(d-a_0)}{b+a_0-d}\le 1+\ep
$$
for some $b\ge 5d$  large enough which we consider fixed from now on. On the other hand,
$$
 \frac{V_+(y,t)}{V_-(y,t)}=\textrm{e}^{\frac{b|y|}{\alpha t}}\left(1-\frac{2|y|}{b+|y|}\right)\le \textrm{e}^{\frac{b(d-a_0)}{\alpha T}}\left(1-\frac{2a_0}{b+d-a_0}\right)<1
$$
for all $t\ge T_1$  if $T_1=T_1(b,\alpha, d,a_0)$ is large enough.  Then, since $J$ is nonincreasing in $\R_+$,
$$
\begin{array}{rcl}
\mathcal{J}&\geq& \displaystyle J(x+a_0) \left(  \int_{-a_0}^{a_0}  V_+(y,t)\, dy +\int_{x-d}^{-a_0}  ( V_+(y,t) -V_- ( y,t))\, dy\right)\\[10pt]
&\ge&\displaystyle J(x+a_0)\left(
 \int_{-a_0}^{a_0}  V_+(y,t)\, dy- \ep\int_{x-d}^{-a_0}  V_+(y,t)\, dy\right).
\end{array}
$$
At this point we observe that for the values of $x$ and $y$ under consideration we have $|y-x|\le d$ and $|x|\le d-a_0$. Hence, $(y+b)^2\le (|y-x|+|x|+b)^2\le (2d-a_0+b)^2$, and therefore
$$
\partial_y V_+(y,t)=\frac{C\textrm{e}^{-\frac{(y+b)^2}{4\alpha t}}}{t^{3/2}}\left(1-\frac{(y+b)^2}{2\alpha t}\right)\le 0
$$
for all $t\ge T_2$ if $T_2=T_2(b,\alpha, d,a_0)$ is large enough. Thus,
choosing $\ep\le\frac{2a_0}{d-2a_0}$,
we conclude that
$$
\mathcal{J}
\ge
 V_+(-a_0,t)\left(2a_0-\ep(d-2a_0)\right) \ge 0, \qquad t\ge T:=\max\{T_1,T_2\}.
$$
\end{proof}

\begin{prop}\label{supersolution} Under the assumptions of Theorem~\ref{thm:main}, there exist $b$, $\alpha$, $T$ and $C$ such that
\begin{equation}\label{local-decay}
u(x,t)\le C\frac{|x|+b}{t^{3/2}}\textrm{e}^{-\frac{(|x|+b)^2}{4\alpha t}}\quad\mbox{in}\quad(\R\setminus\H)\cap{\mathcal{A}}_{\alpha,b,T}.
\end{equation}
\end{prop}

\begin{proof} Take  $b$, $\alpha$ and $T\ge d/(4\alpha)$ such that $V$ is a supersolution of $\partial_t V= LV$ in $(\R\setminus\H)\cap{\mathcal{A}}_{\alpha,b,T}$; see~Lemma~\ref{super}.
Now, since $\frac{|x|+b}{T^{3/2}}\textrm{e}^{-\frac{(|x|+b)^2}{4\alpha T}}\ge bT^{-3/2}\textrm{e}^{-1}$ in  $\{(|x|+b)^2\le 4\alpha T\}$, there exists $C>0$ such that $V(x,T)\ge u(x,T)$ in this set. On the other hand, $u(x,t)\le K t^{-1}$ in $\R\times\R_+$; see~Theorem~\ref{order}. Therefore, if $C$ is large enough, $V(x,t)\ge u(x,t)$ for $4\alpha t\le (|x|+b)^2\le 8\alpha t\le 4\alpha t+d$, $t\ge T$. Since moreover $V\ge0$ everywhere, and in particular in $\H$, the result follows from the comparison principle.
\end{proof}

\section{Far field limit}
\label{sect:far.field}
\setcounter{equation}{0}

This section is devoted to obtaining the large time behavior in the far field scale, $|x|\sim \xi t^{1/2}$.
We prove the result only for $x\in\R_+$, the case of  $\R_-$ being completely analogous.
The proof is much more involved than the one for the problem posed on the half-line, studied in \cite{CEQW2}. Indeed,  now we do not have explicit super and subsolutions with the same asymptotic behavior. We will use a scaling argument instead, an idea which was already used for this purpose in~\cite{TW}.

Let $a$ as in~\eqref{hypotheses.H}. For any $\l>0$ we define
\begin{equation*}\label{ul}
\ul(x,t)=\l^2 u(a+\l x, \l^2 t).
\end{equation*}
The scaled solution satisfies
$$
\partial_t \ul=L_\l\ul\quad\text{for }x\in(\R\setminus\mathcal{H}_a^\l),\ t>0, \qquad \H_a^\l=\{x: a+\l x\in \H\},
$$
where  $L_\l$ is the operator defined by
$$
L_\l \varphi(x)=\l^2\int_\R J_\l(x-y)\big(\varphi(y)-\varphi(x)\big)\,dy,\qquad J_\l(x)=\l J(\l x).
$$
If $\varphi\in C^\infty_{\rm c}(\R)$,  an easy computation which uses the symmetry of the kernel plus Taylor's expansion shows that $L_\l\varphi$ converges uniformly to  $\a\Delta\varphi$ as $\l\to\infty$. As we will see, that is the reason why in the far field scale the asymptotic behavior is related to the local heat equation and not to our original nonlocal problem.

As an immediate consequence of Theorem~\ref{order}, we know that
\begin{equation}\label{estimate_ul}
0\le \ul(x,t)\le C t^{-1} \quad\mbox{for every } x\in\R,\  t>0.
\end{equation}
Moreover, Proposition \ref{supersolution} implies a decay in terms of $\l$ on small spatial sets. More precisely, for every $D>0$, $t_0>0$ there are constants $C$ and $\lambda_0$ such that
\begin{equation}\label{second-estimate-ul}
\ul(x,t)\le C\l^{-1} t^{-3/2},\qquad
|x|\le D/\l,\ t\ge t_0,\ \l\ge\l_0.
\end{equation}
Notice also that Corollary~\ref{corollary:estimates.momenta} gives
\begin{equation}
\label{eq:estimate.mass.lambda}
\int_\R \ul(y,t)\,dy=\lambda M(\lambda^2 t)\le C t^{-1/2}.
\end{equation}

The above size estimates allow us to obtain convergent subsequences. It turns out that the limit is continuous, though the functions $u^\l$ are not.

\begin{prop}\label{derivative-estimates} Under the assumptions of Theorem~\ref{thm:main}, for every sequence $\{u^{\lambda_n}\}_{n=0}^\infty$ such that $\lim\limits_{n\to\infty}\l_n=\infty$ there is a subsequence $\{u^{\l_{n_k}}\}_{k=0}^\infty$ that converges uniformly  in compact subsets of $\R_+\times\R_+$ to a function $\bar u\in C(\R_+\times\R_+)$.
\end{prop}
\begin{proof} In order to simplify notations, we will drop the subscript $n$ when no confusion arises.

As in \cite{CEQW}, we use that $u$ is the solution of \eqref{problem} if and only if it is the solution to the Cauchy problem
$$
\partial_t u-Lu=-\X_{\H}(J*u)\quad\mbox{in }\R\times\R_+,\qquad u(x,0)=u_0(x),\quad x\in\R.
$$
Let  $W_\l(x,t)=\l W(\l x,\l^2 t)$, where
$W$ is the regular part of the fundamental solution to the nonlocal heat equation operator; see~\eqref{fund-sol}. Using the variations of constants formula, we get $\ul(x,t)=\sum_{i=1}^4 v_i^\lambda(x,t)$ for all $t\ge t_0$,
where
$$
\begin{array}{l}
\displaystyle v_1^\lambda(x,t)=\textrm{e}^{-\l^2(t-t_0)}\ul(x,t_0),\\
\displaystyle v_2^\lambda(x,t)=-\l^2\int_{t_0}^t\textrm{e}^{-\l^2(t-s)}
\X_{\H_a^\l}(x)\big(J_\l
*\ul(\cdot,s)\big)(x)\,ds,\\
\displaystyle v_3^\lambda(x,t)=\int_\R W_\l(x-y,t-t_0)\ul(y,t_0)\,dy \\
\displaystyle v_4^\lambda(x,t)=-\l^2\int_{t_0}^t\int_{\H_a^\l} W_\l(x-y,t-s)\big(J_\l*\ul(\cdot,s)\big)(y)\,dy\,ds.
\end{array}
$$
If $x>0$, then $\X_{\H_a^\l}(x)\equiv0$, hence $v_2^\lambda(x,t)=0$. As for $v_1^\lambda$, if $t_0>0$,
$$
0\le \sup_{x>0,\, t\ge 2t_0}v_1^\lambda(x,t)\le C t_0^{-1}\textrm{e}^{-\l^2t_0}\to0\quad\text{as }\l\to\infty.
$$

We now turn our attention to $v_3^\l$.  Let $\bar v^\l$ be the solution of the heat equation with diffusivity $\a$ in $\R\times(t_0, \infty)$ and initial condition $\ul(x,t_0)$,
$$
\bar v^\l(x,t)=\int_\R\Gamma_\a(x-y,t-t_0)\ul(y,t_0)\,dy.
$$
Using the scaling property  $\Gamma_\a(x,t)=\lambda\Gamma_\a(\l x,\l^2 t)$ and the mass estimate~\eqref{eq:estimate.mass.lambda}, we get
$$
|v_3^\l(x,t)-\bar v^\l(x,t)|\le
\l\|W\big(\cdot,\l^2(t-t_0)\big)-\Gamma_\a\big(\cdot,\l^2(t-t_0)\big)\|_{L^\infty(\R)} Ct_0^{-1/2}.
$$
Hence, using estimate~\eqref{estima-W} with $s=0$, we get for $t\ge 2t_0$,
$$
\sup_{x\in\R,\, t\ge 2t_0}|v_3^\l(x,t)-\bar v^\l(x,t)|\le \frac{C}{\lambda t_0^{3/2}}\to0\quad\text{as }\lambda\to\infty.
$$
Thus, compactness for $\{v_3^\l\}$ will follow from compactness for $\{\bar v^\l\}$. But this is a consequence of the well-known regularizing effect for the heat equation, since the initial data $\{u^\l(\cdot,t_0)\}$ are uniformly (in $\lambda$) bounded in $L^\infty(\R)\cap L^1(\R)$; see formulas~\eqref{estimate_ul} and~\eqref{eq:estimate.mass.lambda}. Since the functions $\bar v^\l$ are continuous for $t>t_0$, the same is true for the limit.

We finally consider $v_4^\l$.  We depart from
$$
\left|\partial_x v_4^\l(x,t)\right|\le \l^2\int_{t_0}^t\int_{\H_a^\l}
\left|\partial_x W_\l(x-y,t-s)\right|\big(J_\l*\ul(\cdot,s)\big)(y)\,dy\,ds.
$$
Let $y\in\H_a^\l$. This implies $|y|\le 2a/\l$, hence $\l|\H_a^\l|\le C$. If moreover $|z-y|<d/\l$, then  $|z|\le C/\l$. Thus, $\ul(z,s)\le C\l^{-1}s^{-3/2}$; see estimate~\eqref{second-estimate-ul}. Therefore,  $\big(J_\l*\ul(\cdot,s)\big)(y)\le C\l^{-1}s^{-3/2}$  for $y\in\H_a^\l$. Combining this with the pointwise estimate for $\partial_x W$ in~\eqref{eq:estimates.W}, we obtain, for $x\ge\delta>0$, $2t_0\le t\le T$ and $\l\ge4a/\delta$,
$$
\left|\partial_x v_4^\l(x,t)\right|
\le C\l\int_{t_0}^t\int_{\H_a^\l}\frac{t-s}{|x-y|^4}\,s^{-3/2}\,dy\,ds\le \frac {CT^2\l|\H_a^\l|}{\delta^4t_0^{3/2}}\le C_{\delta,t_0,T,a}.
$$
A similar computation, using the estimate $|\partial_t W|\le  Ct/|x|^5$  for $t\ge t_0>0$ \cite{TW}, shows that $\left|\partial_t v_4^\l(x,t)\right|\le C_{\delta,t_0,T,a}$. The conclusion then follows from Ascoli-Arzela's Theorem. Since the functions $v_4^\l$ are continuous, the same is true for the limit.
\end{proof}

We now identify the limit of any sequence $\{u^{\lambda_n}\}_{n=0}^\infty$ converging to a continuous function in terms of the dipole solution to the local heat equation with diffusivity $\a$. We will  prove that this limit does not depend on the particular sequence. As a corollary, the whole family $\{\ul\}_{\lambda\in\R}$ converges to this common limit.
\begin{prop}
\label{prop:identification}
Under the assumptions of Theorem~\ref{thm:main}, if  $\{u^{\l_n}\}_{n=0}^\infty$, $\lim\limits_{n\to\infty}\l_n=\infty$, converges  uniformly on compact subsets of $\R_+\times\R_+$ to a function $\bar u\in C(\R_+\times\R_+)$, then
$$
\bar u(x,t)=-2 \bar M_1^+ \mathcal{D}_\a(x,t),\qquad \bar M_1^+=\int_0^\infty u_0(x)\phi_+(x)\,dx.
$$
\end{prop}
\begin{proof}For simplicity, we will drop the index $n$ whenever no confusion arises.

We start by studying the trace of $\bar u$ at $x=0$. From estimate~\eqref{local-decay} we get
$$
0\le \ul(x,t)\le C\frac{|x+a/\l|+b/\l}{t^{3/2}} \quad\text{if }\left(\left|x+\frac a\lambda\right|+\frac b\lambda\right)^2\le 4\alpha t,\, t>T/\lambda^2.
$$
Therefore $0\le\bar u(x,t)\le C\frac{x}{t^{3/2}}$ if $0<x\le (\alpha t)^{1/2}$ and $t>0$. In particular, $\lim\limits_{x\to0^+} \bar u(x,t)=0$.

In order to identify the limit it is convenient to work with  a weak notion of solution, since then the stated compactness is enough to pass to the limit. Let $\varphi\in C_{\rm c}^\infty(\R\times\overline\R_+)$ such that $\varphi(0,t)=0$ for all $t\ge0$.
Using the uniform convergence of the family $\{u^\lambda\}$,
$$
\begin{array}{l}
\displaystyle\int_0^\infty\int_0^\infty\bar u(\partial_t\varphi+\a\Delta\varphi)\,dx\,dt
=\underbrace{\int_0^\delta\int_0^\infty\bar u(\partial_t\varphi+\a\Delta\varphi)\,dxdt}_{\mathcal{A}}\\
\displaystyle\qquad\qquad+\lim_{\l\to\infty}\underbrace{\int_\delta^T\int_0^\infty\ul(\partial_t\varphi
+L_\l\varphi)\,dxdt}_{\mathcal{B}}+
\lim_{\l\to\infty}\underbrace{\int_\delta^T\int_0^\infty\ul(\a\Delta\varphi -L_\l\varphi)\,dxdt}_{\mathcal{C}}.
\end{array}
$$

Let us begin with an estimate for $\mathcal{A}$. Since
\[
\begin{aligned}
\int_0^\delta\int_0^\infty\ul(x,t)|\varphi_t+\a\Delta\varphi|&\le C\int_0^\delta\l^2\int_0^\infty u(a+\l x,\l^2 t)\,dx\,dt=C\int_0^\delta\l\int_a^\infty u(y,\l^2 t)\,dy\,dt\\
&\le C\l\int_0^\delta M(\l^2 t)\,dt\le C\l \int_0^\delta (\l^2t)^{-1/2}\,dt=C \delta^{1/2},
\end{aligned}
\]
by applying Fatou's Lemma we get $|\mathcal{A}|\le   C \delta^{1/2}$.

To estimate $\mathcal{B}$, we write it as
$$
\begin{array}{rcl}
\mathcal{B}&=&
-\underbrace{\int_\delta^T\int_0^\infty\Big(\partial_t\ul-L_\l\ul\Big)\varphi}_{\mathcal{B}_1}
-\underbrace{\int_0^\infty\ul(x,\delta)\varphi(x,\delta)\,dx}_{\mathcal{B}_2}\\
&&+\underbrace{\int_0^\infty \ul(x,t)\,L_\l\varphi(x,t)\,dx-\int_0^\infty\varphi(y,t) L_\l\ul(y,t)\,dy}_{\mathcal{B}_3}.
\end{array}
$$
Since $\H_a^\l\subset(-\infty,0)$, then $\partial_t\ul-L_\l\ul=0$ in $x>0$. Hence $\mathcal{B}_1=0$.
As for $\mathcal{B}_2$,  we decompose it as
$$
\begin{array}{rcl}
\mathcal{B}_2&=&\displaystyle \underbrace{\int_0^\infty\ul(x,\delta)(\varphi(x,\delta)-\varphi(x,0))\,dx}_{\mathcal{B}_{21}}
+\underbrace{\int_0^\infty x\ul(x,\delta)\Big(\frac{\varphi(x,0)}x-\partial_x\varphi(0,0)\Big)\,dx}_{\mathcal{B}_{22}}\\
&&+\underbrace{\partial_x\varphi(0,0)\int_0^\infty x\ul(x,\delta)\,dx}_{\mathcal{B}_{23}}
\end{array}
$$
On the one hand, $
|\mathcal{B}_{21}|\le
C\|\varphi(\cdot,\delta)-\varphi(\cdot,0)\|_{L^\infty(\R_+)}\l M(\l^2\delta) \le C \delta^{1/2}$.
On the other hand, since $\left|\frac{\varphi(x,0)}x-\partial_x\varphi(0,0)\right|<C |x|$, using the  estimate for the second momentum  in Corollary~\ref{corollary:estimates.momenta},
$$
\begin{array}{rcl}
|\mathcal{B}_{22}|&\le& \displaystyle C\int_0^\infty x^2 u^\lambda(x,\delta)\,dx=C\l^{-1}\int_a^\infty (y-a)^2u(y,\l^2\delta)\,dy
\\
&\le&\displaystyle  C\l^{-1} M_2(\l^2\delta)\le\displaystyle C\l^{-1}\left(M_2(0)+C\lambda \delta^{1/2}\right)=\displaystyle O(\l^{-1})+O(\delta^{1/2}).
\end{array}
$$
Finally,
$$
\begin{array}{rcl}
\mathcal{B}_{23}&=&\displaystyle \partial_x\varphi(0,0)\left(\int_a^\infty yu(y,\l^2\delta)\,dy-a\int_a^\infty u(y,\l^2\delta)\,dy\right)\\
&=&\displaystyle \partial_x\varphi(0,0)\left(
M_1^+(\lambda^2\delta)-\int_0^a yu(y,\l^2\delta)\,dy-a\int_a^\infty u(y,\l^2\delta)\,dy\right)\\
&=&\displaystyle \partial_x\varphi(0,0)\bar M_1^++O(\l^{-1}\delta^{-1/2}).
\end{array}
$$
Summarizing,
$$
\limsup_{\lambda\to\infty}|\mathcal{B}_2- \partial_x\varphi(0,0)\bar M_1^+|\le C\delta^{1/2}.
$$

We now turn our attention to $\mathcal{B}_3$. Since $J_\lambda\in L^1(\R)$, $\ul(\cdot,t)\in L^\infty(\R)$ and $\varphi(\cdot,t)\in L^1(\R)$, we may  apply Fubini's Theorem in the spatial variable to get, using also the symmetry of the kernel,
$$
\begin{array}{l}
\mathcal{B}_3=\displaystyle \l^2\left(\int_{-\infty}^0\int_0^\infty J_\l(x-y)\ul(x,t)\varphi(y,t)\,dx\,dy-\int_0^\infty \int_{-\infty}^0 J_\l(x-y)\ul(x,t)\varphi(y,t)\,dx\,dy\right)\\
=\displaystyle \l^2\left(\int_{-\frac{d}{\l}}^0\int_0^{\frac{d}{\l}} J_\l(x-y)\ul(x,t)\varphi(y,t)\,dx\,dy-\int_0^{\frac{d}{\l}} \int_{-\frac{d}{\l}}^0 J_\l(x-y)\ul(x,t)\varphi(y,t)\,dx\,dy
\right).
\end{array}
$$
Moreover, the conditions on  $\varphi$ guarantee that $|\varphi(y,t)|\le C|y|$.  Therefore, using~\eqref{second-estimate-ul},
$$
|\mathcal{B}_3|\le \l^2C\l^{-1}t^{-3/2}\int_{-\frac{d}{\l}}^{\frac{d}{\l}}|\varphi(y,t)|\,dy
\le C \l t^{-3/2}\int_{-\frac{d}{\l}}^{\frac{d}{\l}}|y|\,dy\le C t^{-3/2}\l^{-1}.
$$

Finally, since $\ul=O(\delta^{-1})$ for $t\ge \delta$, then
$\lim_{\l\to\infty}\mathcal{C}=0$.

Gathering all the above estimates, we get
\[
\Big|\int_0^\infty\int_0^\infty\bar u(\partial_t\varphi+\a\Delta\varphi)\,dx\,dt+\bar M_1^+\partial_x\varphi(0,0)\Big|\le C\delta^{1/2}.
\]
Since this inequality holds for every $\delta>0$,
\[\int_0^\infty\int_0^\infty\bar u(\partial_t\varphi+\a\Delta\varphi)\,dx\,dt=-\bar M_1^+\partial_x\varphi(0,0).\]

Let $v$ be the antisymmetric extension in the $x$ variable of $\bar u$ to the whole real line.  Then, $v\in C(\R\times\R_+)$ and $v(0,t)=0$ for $t>0$. Let $\psi\in C_{\rm c}^\infty(\R\times\overline\R_+)$ and $\varphi(x,t)=\psi(x,t)-\psi(-x,t)$. Observe that $\varphi(0,t)=0$. Then,
$$
\begin{array}{rcl}
\displaystyle \int_0^\infty\int_\R v(x,t)\big(\partial_t\psi+\a\Delta \psi\big)(x,t)\,dx\,dt&=& \displaystyle \int_0^\infty\int_0^\infty
\bar u(x,t)\,\big(\partial_t\varphi+\a\Delta\varphi\big)(x,t)\,dx\,dt\\
&=& \displaystyle
-\bar M_1^+\partial_x\varphi(0,0)=-2\bar M_1^+\partial_x\psi(0,0);
\end{array}
$$
that is, $v$ is a solution to the local heat equation with diffusivity $\a$ and initial datum $-2\bar M^+_1\delta'$, with $\delta'$ the derivative of the Dirac mass. Hence, $v=-2\bar M_1^+\mathcal{D}_\a$, and the result follows.
\end{proof}

The above results, conveniently rewritten, yield the main result of this section, the far field limit.

\begin{teo}\label{far field-limit} Under the hypotheses of Theorem~\ref{thm:main}, for every $0<\delta<D<\infty$,
\begin{equation*}\label{convergence}
\sup_{x\in(a+\delta t^{1/2},a+Dt^{1/2})}
t|u(x,t)+2\bar M_1^+\mathcal{D}_\a(x,t)|\to0\quad\mbox{as }t\to\infty.
\end{equation*}
\end{teo}

\begin{proof}
We have
$$
u^\l(y,1)+2\bar M_1^+\mathcal{D}_\a^\l(y,1)=\underbrace{u^\l(y,1)+2\bar M_1^+\mathcal{D}_\a(y,1)}_{\mathcal{A}^\l(y)}
+2\bar M_1^+\underbrace{\left(\mathcal{D}_\a^\l(y,1)-\mathcal{D}_\a(y,1)\right)}_{\mathcal{B}^\l(y)}.
$$
As a consequence of Propositions~\ref{derivative-estimates} and~\ref{prop:identification}, we have $\sup_{y\in[\delta,D]}|\mathcal{A}^\l(y)|\to0$ as $\l\to\infty$. On the other hand,
a straightforward computation shows that $|\mathcal{B}^\l(y)|=|\mathcal{D}_\a(y+a/\l,1)-\mathcal{D}_\a( y,1)|\le C \l^{-1}$.
Therefore, as $\l\to\infty$,
$$
\sup_{y\in[\delta,D]}\left|u^\l(y,1)+2\bar M_1^+\mathcal{D}_\a^\l(y,1)\right|=\l^2\sup_{y\in[\delta,D]}\left|
u(a+\lambda y,\lambda^2)+2\bar M_1^+\mathcal{D}_\a(a+\lambda y, \lambda^2)\right|\to0.
$$
Hence, the result follows just renaming $a+\l y$ as $x$  and $\l^2$ as $t$.
\end{proof}

Since $a+\frac\delta2t^{1/2}\le \delta t^{1/2}$ for $t\ge (2a/\delta)^2$, using the behavior of $\phi_0$ at $\pm\infty$
we have, as a corollary of this theorem and the corresponding one for $\R_-$, the following result.

\begin{coro}
\label{cor:far field.limit}
Under the hypotheses of Theorem~\ref{thm:main}, for every $0<\delta<D<\infty$,
\begin{equation*}
\sup_{\delta t^{1/2}\le |x|\le D t^{1/2}}t\left|u(x,t)+2\frac{\phi_0(x)}{x}\mathcal{D}_\a(x,t)\right|\to0\quad\mbox{as }t\to\infty.
\end{equation*}
\end{coro}


\section{Near field limit and global approximant}
\label{sect:near.field}
\setcounter{equation}{0}

This section is devoted to completing the proof of
Theorem~\ref{thm:main}. Since $\phi_0(x)/(|x|+1)$ is bounded and $\mathcal{D}_{\mathfrak q}(x,t)=-\frac{x}{2{\mathfrak q}t}\Gamma_{\mathfrak q}(x,t)$, the proof will follow from the next result, if we take into account~\eqref{estima-W} with $s=0$.
\begin{teo}\label{thm:main.W}
Under the assumptions of Theorem~\ref{thm:main},
\begin{equation*}
\label{eq:main.thm.W}
\sup_{x\in\R}\left(\frac{t^{3/2}}{|x|+1}\left|u(x,t)-\frac{\phi_0(x)W(x,t)}{\a t}\right|\right)\to 0
 \quad\text{as }t\to\infty.
\end{equation*}
\end{teo}
The advantage of this formulation in terms of $W$ is that  it is more straightforward to apply the nonlocal operator $L$ to $W(x,t)/t$ than to $\mathcal{D}_\a(x,t)/x$.

We already know that the result is true in the far field scale; see~Corollary~\ref{cor:far field.limit}. The next step is to prove it for the near field scale. This is done through comparison in  $|x|\le \delta t^{1/2}$, $\delta$ small, with  suitable barriers $w_\pm$ approaching (fast enough) the asymptotic limit as $t$ goes to infinity,
\begin{equation}
\label{eq:def.wpm}
w_\pm(x,t)=\underbrace{\frac{\phi_0(x)W(x,t)}{\a t}}_{v(x,t)}\pm K_\pm R(x,t), \qquad K_\pm\ge1,\quad \lim_{t\to\infty}t^{3/2}\sup_{x\in\R}\frac{|R(x,t)|}{|x|+1}=0.
\end{equation}

We start by estimating how far is $v$ from being a solution. We need a more precise bound than the one obtained in \cite{CEQW2} for the half-line.
\begin{lema} \label{lema-v}
Assume~\eqref{hypotheses.H} and~\eqref{hypotheses.J}. For every $D>0$ there exists $c>0$ such that,
\begin{equation}
\label{eq:estimate.error.v}
\left|\partial_t v-Lv\right|(x,t)\le c t^{-\frac{12}5},\qquad |x|\le Dt^{1/2},\, x\not\in\H,\ t\ge1.
\end{equation}
\end{lema}
\begin{proof} We assume that $x\ge0$. The case $x\le0$ is treated in a similar way.
On the one hand,
$$
\begin{array}{l}
\displaystyle \partial_t v(x,t)=-\frac{\phi_0(x)}{t^2}W(x,t)+\frac{\phi_0(x)}{t}\partial_t W(x,t),\\
\displaystyle Lv(x,t)=\frac{\phi_0(x)}tLW(x,t)+\frac1t\int_\R J(x-y)\big(\phi_0(y)-\phi_0(x)\big)\big(W(y,t)-W(x,t)\big)\,dy.
\end{array}
$$
Therefore, using the equation satisfied by $W$, see~\eqref{eq-W},  we have
$$
\begin{array}{l}
(\partial_t v-Lv)(x,t)=-\underbrace{\frac{\phi_0(x)}{t^2}W(x,t)}_{\mathcal{A}}
 -\underbrace{\frac1t\int_\R J(x-y)\big(\phi_0(y)-\phi_0(x)\big)\big(W(y,t)-W(x,t)\big)\,dy}_{\mathcal{B}}\\
\qquad\qquad\qquad\qquad+\underbrace{\textrm{e}^{-t}J(x)\frac{\phi_0(x)}{t}}_{\mathcal{C}}.
\end{array}
$$

Since $|\phi_0(x)-\bar M_1^+ x|\le C$, see Lemma~\ref{lemma:estimates.derivatives.phi}, using~\eqref{estima-W} with $s=0$ we obtain
$$
\begin{aligned}
\mathcal{A}
=&\frac{\bar M_1^+x}{t^2}\Gamma_\a(x,t)+\frac{\phi_0(x)-\bar M_1^+x}{t^2}\Gamma_\a(x,t)+
\frac{(\phi_0(x)-\bar M_1^+x)+\bar M_1^+x}{t^2}\left(W(x,t)-\Gamma_\a(x,t)\right)\\
=&\frac{\bar M_1^+x}{t^2}\Gamma_\a(x,t)+O(t^{-5/2})\quad\mbox{if } 0<x<Dt^{1/2}.
\end{aligned}
$$

In order to estimate $\mathcal{B}$, we decompose it as
$$
\begin{array}{rcl}
\displaystyle \mathcal{B}&=& \displaystyle
\underbrace{\frac1t\int_\R J(x-y)(y-x)\big(\phi_0(y)-\phi_0(x)\big)\int_0^1( \partial_x W(x+s(y-x),t)-\partial_x W(x,t))\,ds\big)\,dy}_{\mathcal{B}_1}\\
&&\displaystyle +\underbrace{\frac1t \partial_x W(x,t)\int_\R J(x-y)(y-x)\big(\phi_0(y)-\phi_0(x)\big)\,dy}_{\mathcal{B}_2}.
\end{array}
$$
Using formula~\eqref{decaimiento-W} with $s=2$, we get $\mathcal{B}_1=\displaystyle O(t^{-5/2})$. As for $\mathcal{B}_2$, we write it as
$$
\begin{array}{l}
\displaystyle \mathcal{B}_2=\displaystyle \underbrace{\frac1t \partial_x W(x,t)g(x)}_{\mathcal{B}_{21}}+\underbrace{\frac1t \partial_x W(x,t)\phi_0'(x)
\int_\R J(x-y)(y-x)^2\,dy}_{\mathcal{B}_{22}},
\\
\displaystyle g(x)=\int_\R J(x-y)(y-x)^2\int_0^1\big(\phi_0'(x+s(y-x))-\phi_0'(x)\big)\,dsdy.
\end{array}
$$
Lemma~\ref{lemma:estimates.derivatives.phi} implies that $g(x)\le C((1+|x|)^{-1})$. Then, using formula~\eqref{estima-W} with $s=1$,
$$
\displaystyle \mathcal{B}_{21}=\displaystyle \frac1t \partial_x\Gamma_\a(x,t)g(x)+\frac1t(\partial_x W(x,t)-\partial_x\Gamma_\a(x,t))g(x)
=O(t^{-5/2}).
$$
As for the other term, since  $|\phi_0'(x) |\le C$ if $x\not\in\H$, using  again formula~\eqref{estima-W} with $s=1$,
$$
\displaystyle \mathcal{B}_{22}=\displaystyle \frac{2\a}t \partial_x \Gamma_\a(x,t)\phi_0'(x)+\frac{2\a}t(\partial_x W(x,t)-\partial_x\Gamma_\a(x,t))\phi_0'(x)\\
=\displaystyle -\frac x{t^2}{\Gamma_\a}(x,t)\phi_0'(x)+O(t^{-5/2}).
$$

Finally, since $0\le J(x)\phi_0(x)\le C$, we have $|\mathcal{C}|\le C t^{-1}\textrm{e}^{-t}$.

Gathering all these estimates,
\begin{equation*}\label{estimate-v}
|\partial_t v-Lv|(x,t)\le O(t^{-5/2})-\frac{x\Gamma_\a(x,t)}{t^2}\left(\bar M_1^+- \phi_0'(x)\right).
\end{equation*}
We  now use that $|\phi_0'(x)-\bar M_1^+|\le C/|x|^{4/5}$, see Lemma~\ref{lemma:estimates.derivatives.phi}, to obtain,
$$
|\partial_t v-Lv|(x,t)\le Ct^{-5/2}+Ct^{-5/2}x^{1/5}\le c t^{-\frac52+\frac1{10}}\quad \mbox{if }0<x\le Dt^{1/2},\ x\notin\H.
$$
\end{proof}

We now turn our attention to the correction term. We consider a function $R$ of the form
\begin{equation}
\label{eq:def.R}
 R(x,t)=\begin{cases}\left((|x|+d)^\gamma+k\right) t^{-\frac{3+\kappa}2}\quad&\mbox{if }|x|\ge a_0,\\
 0\quad&\mbox{if } |x|< a_0.
 \end{cases}
\end{equation}
Notice that if the parameters are conveniently chosen, $R$ will decay in the region $|x|\le \delta t^{1/2}$ in the desired way. In order to show that $w_+$ is a supersolution and $w_-$ a subsolution, we need to estimate the action of the diffusion operator on $R$ from below. This is  done next.

\begin{lema}
\label{lem:R}
Assume~\eqref{hypotheses.J}.
Given  $0<\kappa,\gamma<1$, there are values $\delta\in(0,1)$ and $k>0$  such that the function $R$ defined in~\eqref{eq:def.R} satisfies
\begin{equation}\label{eq-R}
(\partial_t R-LR)(x,t)\ge \frac\a8\gamma(1-\gamma)(|x|+d)^{\gamma-2} t^{-\frac{3+\kappa}2}\quad\mbox{for } a_0\le |x|\le \delta t^{1/2},\, t\ge (d/\delta)^2.
\end{equation}
\end{lema}

\begin{proof} We assume that $x>a_0$. The case $x<-a_0$ is done similarly.

A straightforward computation yields
$$
\partial_t R(x,t)=-\frac{3+\kappa}2 \Big((|x|+d)^\gamma+k\Big)t^{-\frac{5+\kappa}2}.
$$
On the other hand, if $x-d\ge -a_0$, Taylor's expansion plus the symmetry of the kernel produce
$$
\begin{array}{rcl}
\displaystyle t^{\frac{3+\kappa}2}LR(x,t)&=&\displaystyle
\int_{\max\{x-d,a_0\}}^{x+d}J(x-y)\big((y+d)^\gamma+k\big)\,dy-(x+d)^\gamma-k\\
&=&\displaystyle \int_{x-d}^{x+d}J(x-y)\big((y+d)^\gamma-(x+d)^\gamma\big)\,dy
-\int_{x-d}^{\max\{x-d,a_0\}}J(x-y)(y+d)^\gamma\,dy \\
&&\displaystyle -k\int_{x-d}^{\max\{x-d,a_0\}}J(x-y)\,dy
\le \a\gamma(\gamma-1)(\xi+d)^{\gamma-2},
\end{array}
$$
for some  $\xi\in(x-d,x+d)$. Notice that $\xi+d>0$. Moreover,  $(x+d)/(\xi+d)\ge1/2$. Therefore, since $\gamma\in(0,1)$, we conclude that
$$
LR(x,t)\le\displaystyle
\frac\a4\gamma(\gamma-1)(x+d)^{\gamma-2}t^{-\frac{3+\kappa}2}.
$$
If $x-d<-a_0$, which is only possible if $d>2a_0$,
$$
\begin{array}{l}
\displaystyle t^{\frac{3+\kappa}2}LR(x,t)\\
\displaystyle\qquad=\int_{a_0}^{x+d}J(x-y)\big((y+d)^\gamma+k\big)\,dy
+\int_{x-d}^{-a_0}J(x-y)\big((d-y)^\gamma+k\big)\,dy-(x+r)^\gamma-k\\
\displaystyle\qquad=\int_{x-d}^{x+d}J(x-y)\big((y+d)^\gamma-(x+d)^\gamma\big)\,dy
+\int_{x-d}^{-a_0}J(x-y)\big((d-y)^{\gamma}-(y+d)^\gamma\big)\,dy \\
\displaystyle\qquad\quad-\int_{-a_0}^{a_0}J(x-y)(y+d)^{\gamma}\,dy-k\int_{-a_0}^{a_0}J(x-y)\,dy.
\end{array}
$$
Hence, using~\eqref{eq:nu} and choosing $k=(2d-a_0)^\gamma/\int_{d-2a_0}^dJ(y)\,dy$,
$$
\begin{array}{rcl}
\displaystyle t^{\frac{3+\kappa}2}LR(x,t)&\le& \displaystyle\a\gamma(\gamma-1)(\xi+d)^{\gamma-2}
+\int_{x-d}^{-a_0}J(x-y)(d-y)^\gamma\,dy-k\int_{d-2a_0}^dJ(y)\,dy\\
&\le&  \displaystyle\frac\a4\gamma(\gamma-1)(x+d)^{\gamma-2}.
\end{array}
$$

Since $|x|+d\le 2\delta t^{1/2}$ for $|x|\le \delta t^{1/2}$ and $t\ge (d/\delta)^2$, the above estimates yield
$$
\begin{array}{rcl}
\displaystyle (\partial_t R-LR)(x,t)&\ge& \displaystyle \left( -\frac{3+\kappa}2 (2\delta)^2-k\frac{3+\kappa}2(2\delta)^{2-\gamma}\left(\frac{\delta}{d}\right)^{\gamma}+\frac\a4\gamma(1-\gamma)\right) (x+d)^{\gamma-2}t^{-\frac{3+\kappa}2}\\
&\ge&\displaystyle  \frac\a8\gamma(1-\gamma)(x+d)^{\gamma-2}t^{-\frac{3+\kappa}2}
\end{array}
$$
if we choose $\delta$ small.
\end{proof}

The combination of the two previous lemmas allows to prove that $w_\pm$ are barriers in the region under consideration.
\begin{lema}
\label{lemma.sub.super}
Assume~\eqref{hypotheses.H} and~\eqref{hypotheses.J}. There exists values $0<\kappa,\gamma<1$, $k>0$, $\delta\in(0,1)$ and $t_0>0$  such that,   for every $K_\pm\ge1$, the functions $w_\pm$ defined in~\eqref{eq:def.wpm} satisfy
$$
\partial_t w_+-Lw_+>0,\quad \partial_t w_--Lw_-<0\quad\mbox{if }  |x|\le \delta t^{1/2},\  x\notin\H,\ t\ge t_0.
$$
\end{lema}
\begin{proof} Since $|x|+d\le 2\delta t^{1/2}$ for $|x|\le \delta t^{1/2}$ and $t\ge (d/\delta)^2$, using Lemmas~\ref{lema-v} and~\ref{lem:R} we get
$$
(\partial_t w_+-Lw_+)(x,t)\ge t^{-\frac{5+\kappa-\gamma}2}\left(K_+\frac\a8\gamma(1-\gamma)2^{\gamma-2}\delta^{\gamma-2}-
ct^{-\frac{5\gamma-1-5\kappa}{10}}\right).
$$
where $c$ is the constant in \eqref{eq:estimate.error.v} for $D=1$.
We now choose $\gamma\in(1/5,1)$ and then $\kappa\in (0,\gamma-\frac15)$. The result follows immediately, taking $t_0$ large enough.

The computation for $w_-$ is completely analogous.
\end{proof}


We can now proceed to the proof of Theorem~\ref{thm:main.W}.
It requires a matching with the far field limit and a control of the size of $u$ and the limit function in the very far field.

\begin{proof}[Proof of Theorem~\ref{thm:main.W}]
We already know that  both $tu(x,t)$ and $\phi_0(x)\Gamma_\a(x,t)$ are  bounded for $x\in\R$, $t>1$. Therefore, using~\eqref{estima-W} with $s=0$, for every $\ep>0$ we have a value $D_\ep$ such that
$$
\sup_{|x|\ge D_\ep t^{1/2}}\left|\frac {t^{3/2}}{|x|+1}\left(u(x,t)-\frac{\phi_0(x)}{\a t}W(x,t)\right)\right|\le \frac C{D_\ep}<\ep, \quad t\ge 1,
$$
which gives the result for the very far field. We assume without loss of generality that $D_\ep>2$.

Let $\delta\in (0,1)$ be the value provided by Lemma~\ref{lemma.sub.super}.
Corollary~\ref{cor:far field.limit} implies, using~\eqref{estima-W} with $s=0$, that there exists $t_{\ep,\delta}\ge1$ such that
$$
t\Big|u(x,t)-\frac{\phi_0(x)}{\a t} W(x,t)\Big|<\ep\delta,\quad \delta t^{1/2}\le |x|\le D_\ep t^{1/2}, t>t_{\ep,\delta}.
$$
This means, on the one hand, that $\frac{t^{3/2}}{|x|+1}\Big|u(x,t)-\frac{\phi_0(x)}{\a t}W(x,t)\Big|<\ep$ in such sets. On the other hand, since $\delta+1<D_\ep$,
$$
\frac{\phi_0(x)}{\a t} W(x,t)-\frac{\ep}{t}\le u(x,t)\le \frac{\phi_0(x)}{\a t} W(x,t)+\frac{\ep}{t}, \quad \delta t^{1/2}\le |x|\le(\delta+1) t^{1/2}, t>t_{\ep,\delta}.
$$
Now we notice that there exists a value $\theta=\theta(\delta)>0$ such that $\frac{\phi_0(x)}{\a}W(x,t)\ge\theta$ for all  $\delta t^{1/2}\le |x|\le(\delta+1) t^{1/2}$, $t\ge 1$. Thus, if $\ep<\theta$,
$$
\left(1-\frac\ep\theta\right)\frac{\phi_0(x)}{\a t} W(x,t)\le u(x,t)\le \left(1-\frac\ep\theta\right)\frac{\phi_0(x)}{\a t} W(x,t), \quad \delta t^{1/2}\le |x|\le(\delta+1) t^{1/2}, t>t_{\ep,\delta}.
$$
In particular, we have $\left(1-\frac\ep\theta\right)w_-(x,t)\le u(x,t)\le \left(1+\frac\ep\theta\right)w_+(x,t)$ at the \lq lateral' boundary of the set $|x|\le\delta t^{1/2}$, no matter what the values $K_\pm\ge1$ are. These inequalities are also trivially true at the \lq inner' boundary $\H$.
On the other hand, $R(x,t_{\ep,\delta})\ge (a_0+d)^\gamma t_{\ep,\delta}^{-\frac{3+\kappa}2}>0$. Therefore, since $u(x,t_{\ep,\delta})$ is bounded, if we choose $K_\pm\ge1 $ large enough we have
$$
\left(1-\frac\ep\theta\right)w_-(x,t_{\ep,\delta})\le u(x,t_\ep)\le \left(1+\frac\ep\theta\right)w_+(x,t_{\ep,\delta})\quad
\mbox{for }|x|<\delta t_{\ep,\delta}^{1/2}.
$$
We may then apply the comparison principle to obtain
$$
\left(1-\frac\ep\theta\right)w_-(x,t)\le u(x,t)\le \left(1+\frac\ep\theta\right)w_+(x,t)\quad
\mbox{for }|x|<\delta t^{1/2},\ x\not\in\H,\ t\ge t_{\ep,\delta}.
$$
Thus, using the decay estimate~\eqref{estima-W} with $s=0$, and the fact that $\phi(x)/(1+|x|)$ is bounded,
$$
\frac {t^{3/2}}{|x|+1}\left(u(x,t)-\frac{\phi_0(x)}{\a t}W(x,t)\right)\le
C\ep+\left(1+\frac\ep\theta\right)K_+
\frac{t^{3/2}|R(x,t)|}{|x|+1}
$$
if $|x|<\delta t^{1/2}$, $x\not\in\H$, $t\ge t_{\ep,\delta}$. Letting $t\to\infty$,  we conclude from the decay properties of $R$ that
$$
\limsup_{t\to\infty}\sup_{|x|<\delta t^{1/2},x\not\in\H}\frac {t^{3/2}}{|x|+1}\left(u(x,t)-\frac{\phi_0(x)}{\a t}W(x,t)\right)\le C\ep.
$$
An analogous argument shows that
$$
\liminf_{t\to\infty}\sup_{|x|<\delta t^{1/2},x\not\in\H}\frac {t^{3/2}}{|x|+1}\left(u(x,t)-\frac{\phi_0(x)}{\a t}W(x,t)\right)\ge -C\ep.
$$
\end{proof}

\section{The very far field}
\label{sect:very.far.field}
\setcounter{equation}{0}

In the very far field scale, $|x|\ge g(t)t^{1/2}$ with $\lim_{t\to\infty}g(t)=\infty$, up to now we only know that $u(\cdot,t)=O(t^{-1})$. Actually, we can do better, and prove that $u(\cdot,t)=o(t^{-1})$ in that region.

\begin{teo}
\label{thm:very far field}
Under the hypotheses of Theorem~\ref{thm:main}, if $\lim_{t\to\infty}g(t)=\infty$, then
$$
\sup_{|x|\ge g(t)t^{1/2}}t u(x,t)\to0\quad\mbox{as }t\to\infty.
$$
\end{teo}
\begin{proof} Comparison  with the solution  of the Cauchy problem that has initial datum $u(x,t/2)$ at time $t/2$ yields
$$
u(x,t)\le\textrm{e}^{-t/2}u(x,t/2)+\int_\R W(x-y,t/2)u(y,t/2)\,dy,\qquad t\ge0.
$$
Hence, since $\|W(\cdot,t)-\Gamma_\a(\cdot,t)\|_{L^\infty(\R)}=O(t^{-1})$, see~\eqref{estima-W}, and $u(\cdot,t)=O(t^{-1})$, see Theorem~\ref{order},
$$
tu(x,t)\le C \textrm{e}^{-t/2}+t\int_\R \Gamma_\a(x-y,t/2)u(y,t/2)\,dy+
C\int_\R u(y,t/2)\,dy.
$$
But we already know that the mass decays to 0. Therefore,  the result will follow if we are able to prove that
$$
\sup_{|x|\ge g(t)t^{1/2}}
t\int_\R \Gamma_\a(x-y,t/2)u(y,t/2)\,dy\to0\quad\text{as }t\to\infty.
$$
To this aim, we decompose the integral as
$$
\begin{array}{l}
\displaystyle t\int_\R \Gamma_\a(x-y,t/2)u(y,t/2)\,dy
=\underbrace{t\int_{|y|<\delta t^{1/2}}\Gamma_\a(x-y,t/2)u(y,t/2)\,dy}_{\mathcal{A}}\\
\qquad+\underbrace{
t\int_{\delta t^{1/2}<|y|<Dt^{1/2}}\Gamma_\a(x-y,t/2)u(y,t/2)\,dy}_{\mathcal{B}}+
\underbrace{t\int_{|y|>D t^{1/2}}\Gamma_\a(x-y,t/2)u(y,t/2)\,dy}_{\mathcal{C}}.
\end{array}
$$

Using the estimates for the size and the first momentum, we get
$$
\begin{array}{l}
\displaystyle\mathcal A\le Ct^{-1/2}\int_{|y|<\delta t^{1/2}}\textrm{e}^{-\frac{|x-y|^2}{2\a t}}\,dy\le C\delta,\\
\displaystyle\mathcal C\le t^{1/2}\int_{|y|>D t^{1/2}}\textrm{e}^{-\frac{|x-y|^2}{2 \a t}}u(y,t/2)\,dy\le
\frac{t^{1/2}}{Dt^{1/2}}\int|y|u(y,t/2)\,dy\le \frac CD.
\end{array}
$$
On the other hand,
$$
\begin{array}{l}
\mathcal{B}\le
\underbrace{t\int_{\delta t^{1/2}<|y|<Dt^{1/2}}\Gamma_\a(x-y,t/2)\left|u(y,t/2)+\frac{2\phi_0(y)}{y}\mathcal{D}_\a(y,t/2)\right|\,dy}_{\mathcal{B}_1}\\
\qquad\qquad+
\underbrace{t\int_{\delta t^{1/2}<|y|<Dt^{1/2}}\Gamma_\a(x-y,t/2)\frac{2\phi_0(y)}{y}\left|\mathcal{D}_\a(y,t/2)\right|\,dy}_{\mathcal{B}_2}.
\end{array}
$$
The far field limit plus the Dominated Convergence Theorem yield $\sup_{x\in\R}\mathcal{B}_1\to 0$ as $t\to\infty$. As for the other term, we have for $t$ large,
$$
\mathcal{B}_2\le4\frac{\max\{\bar M_1^+,\bar M_1^-\}}{\a^{3/2}\sqrt{2\pi}}\int_{|y|<Dt^{1/2}}\frac {|y|}{t^{1/2}}\textrm{e}^{-\frac{|y|^2}{2\a t}}\Gamma_\a(x-y,t/2)\,dy\le C\int_{|y|<Dt^{1/2}}\Gamma_\a(x-y,t/2)\,dy.
$$
We now perform the change of variables $z=y-x$. In the set under consideration we have
$|z|\ge (g(t)-D)t^{1/2}$. Hence,
$$
\sup_{|x|\ge g(t)t^{1/2}}\mathcal{B}_2\le C\int_{|z|>(g(t)-D)t^{1/2}}\Gamma_\a(z,t)\,dz\to0 \quad\text{as }t\to\infty.
$$

Summarizing,
$$
\limsup_{t\to\infty}
\sup_{|x|\ge g(t)t^{1/2}}
t\int_\R \Gamma_\a(x-y,t/2)u(y,t/2)\,dy\le C\delta+\frac{C}{D}.
$$
The result follows letting $\delta\to0$, $D\to\infty$.
\end{proof}

\end{document}